\documentclass[11pt,american,preprint]{elsarticle}
\makeatletter
\def\ps@pprintTitle{%
 \let\@oddhead\@empty
 \let\@evenhead\@empty
 \def\@oddfoot{\centerline{\thepage}}%
 \let\@evenfoot\@oddfoot}
\makeatother

\usepackage[margin=1in]{geometry}

\usepackage{xcolor}
\usepackage{lipsum}
\usepackage{changepage}
\usepackage{centernot}

\usepackage[T1]{fontenc}
\usepackage[latin9]{luainputenc}
\usepackage{amsmath}
\usepackage{amsthm}
\usepackage{amssymb}
\usepackage{stmaryrd}
\usepackage{esint}
\usepackage{nameref}
\usepackage{hyperref} 
\usepackage{cleveref} 

\usepackage[shortlabels]{enumitem}
\usepackage{chngcntr}

\usepackage{mathrsfs}

\newsavebox{\foobox}

\usepackage{graphicx,amssymb}

\usepackage{array}
\newcolumntype{M}[1]{>{\centering\arraybackslash}m{#1}}

\usepackage{bbm} 
\makeatletter
\numberwithin{equation}{section}
\theoremstyle{plain}
\newtheorem{thm}{\protect\theoremname}[section]

\theoremstyle{plain*}
\newtheorem*{thm*}{\protect\theoremname}

\theoremstyle{plain}
\newtheorem{lem}[thm]{\protect\lemmaname}
  
\theoremstyle{plain*}
\newtheorem*{lem*}{\protect\lemmaname}  
  
  \theoremstyle{plain}
  \newtheorem{prop}[thm]{\protect\propositionname}
  
    \theoremstyle{plain*}
  \newtheorem*{prop*}{\protect\propositionname}

\theoremstyle{remark}
\newtheorem{question}[thm]{Question}
\theoremstyle{remark*}
\newtheorem*{question*}{Question} 
\theoremstyle{remark}
\newtheorem{rem}[thm]{\protect\remarkname}
\theoremstyle{remark*}
\newtheorem*{rem*}{\protect\remarkname}
\theoremstyle{remark}

\theoremstyle{remark*}
\newtheorem*{example*}{\protect\examplename}

\theoremstyle{plain}
\newtheorem{cor}[thm]{\protect\corollaryname}
\providecommand{\corollaryname}{Corollary}
  
\theoremstyle{definition}
  
\makeatother

\theoremstyle{plain} 
\newcommand{\thistheoremname}{}
\newtheorem{genericthm}[thm]{\thistheoremname}

\newtheorem*{genericthm*}{\thistheoremname}
\newenvironment{namedthm*}[1]
  {\renewcommand{\thistheoremname}{#1}%
   \begin{genericthm*}}
  {\end{genericthm*}}

\makeatother

\usepackage{babel}
 \providecommand{\lemmaname}{Lemma}
  \providecommand{\propositionname}{Proposition}
  \providecommand{\remarkname}{Remark}
\providecommand{\theoremname}{Theorem}

\newcommand{\R}{\mathbb{R}}
\newcommand{\N}{\mathbb{N}}
\newcommand{\Q}{\mathbb{Q}}
\newcommand{\Z}{\mathbb{Z}}

\newcommand\precdot{\mathrel{\ooalign{$\prec$\cr
  \hidewidth\raise0ex\hbox{$\cdot\mkern0.5mu$}\cr}}}
\newcommand\preceqdot{\mathrel{\ooalign{$\preceq$\cr
  \hidewidth\raise0.225ex\hbox{$\cdot\mkern0.5mu$}\cr}}}

\usepackage{fancyhdr}

\usepackage{chngcntr}
\usepackage{apptools}
\AtAppendix{\counterwithin{thm}{section}}
\usepackage{hyperref}

\begin{document}
\begin{frontmatter}
\title{Coexistence of mixing and rigid behaviors in ergodic theory}

\author[add1]{Rigoberto Zelada}
\ead{Rigoberto.Zelada@warwick.ac.uk}

\address[add1]{Mathematics Institute. University of Warwick, Coventry, CV4 7AL, UK}

\begin{abstract}
 In this paper we introduce and explore the notion of \textit{rigidity group}, associated with a collection of finitely many  sequences, and  show that this concept has many, somewhat surprising characterizations of algebraic, spectral, and unitary nature. Furthermore, we demonstrate  that these characterizations can be employed to obtain various results in the theory of generic Lebesgue-preserving automorphisms of $[0,1]$,  IP-ergodic theory, multiple recurrence, additive combinatorics, and spectral theory. 
  As a consequence of one of our results we show that given $(b_1,...b_\ell)\in\N^\ell$, there is no orthogonal vector $(a_1,\dots,a_\ell)\in\Z^\ell$ with some $|a_j|=1$ if and only if  there is an increasing sequence of natural numbers $(n_k)_{k\in\N}$ with the property that for each $F\subseteq \{1,...,\ell\}$ there is a $\mu$-preserving transformation $T_F:[0,1]\rightarrow[0,1]$ ($\mu$ denotes the Lebesgue measure)  such that for any measurable $A,B\subseteq [0,1]$,   
$$\lim_{k\rightarrow\infty}\mu(A\cap T_F^{-b_jn_k}B)=\begin{cases}
\mu(A\cap B),\,\text{ if }j\in F,\\
\mu(A)\mu(B),\,\text{ if }j\not\in F.
\end{cases}$$
We remark that this result has a natural extension to a wide class of families of sequences.
\end{abstract}
\end{frontmatter}
\tableofcontents
\section{Introduction}
In this paper we deal with the concept of \textit{group of rigidity} which, as we explain below, traces back its roots to the work of V. Bergelson, S. Kasjan, and M. Lema{\'n}czyk \cite{BKLUltrafilterPoly} who dealt  with the closely related notions of \textit{N-rigidity group} and \textit{group of global rigidity}. The concept of group of rigidity has found multiple (tacit) applications to IP-ergodic theory \cite{BKLUltrafilterPoly}, the study of generic properties of Lebesgue preserving automorphisms of $[0,1)$ \cite{zelada2023GaussianETDS}, spectral decomposition of unitary operators \cite{BKLUltrafilterPoly}, mixing properties of Gaussian systems \cite{zelada2023GaussianETDS}, and combinatorics \cite{BerZel-NiceRecurrence}. Our goal in this paper is to obtain various "elementary" characterizations for the concept of group of rigidity  and then utilize these results to obtain   applications to ergodic theory and combinatorics.
\subsection{Historical motivation and background}
Let $([0,1],\mathcal B,\mu)$ be the  probability space
where $\mathcal B=\text{Borel}([0,1])$ and $\mu$ is the Lebesgue measure. We denote the set of all invertible $\mu$-preserving transformations by $\text{Aut}([0,1],\mathcal B,\mu)$ and endow it with the so called \textit{weak topology} (see \eqref{5.eq:DefnMetric} for the definition). With this topology, $\text{Aut}([0,1],\mathcal B,\mu)$ becomes  a completely metrisable topological group.\\
Two of the most interesting dynamical  phenomena for $\mu$-preserving transformations are those of \textit{weak mixing} and \textit{rigidity}.\\
Among the many equivalent definitions of weak mixing (see \cite{BergelsonWMSurvey}, for example), we have that a  $\mu$-preserving transformation $T$ is called \textbf{weakly mixing} if there exists an increasing sequence $(n_k)_{k\in\N}$ in $\N$ with the property that for any $A,B\in\mathcal B$,  
$$\lim_{k\rightarrow\infty}\mu(A\cap T^{-n_k}B)=\mu(A)\mu(B).
$$
As proved by Halmos \cite{halmos1944generaMixingl}, the set $\mathcal{W}$ formed by the weakly mixing  $\mu$-preserving transformations $T\in\text{Aut}([0,1],\mathcal B,\mu)$ is a dense $G_\delta$ set.\\
A $\mu$-preserving transformation $T$ is called \textbf{rigid} if there exists an increasing sequence $(n_k)_{k\in\N}$ in $\N$ with the property that for any $A,B\in\mathcal B$, 
$$\lim_{k\rightarrow\infty}\mu(A\cap T^{-n_k}B)=\mu(A\cap B).$$
A "folklore" result (see \cite[Proposition 2.9]{BdJLRArxivrigidity}, for example), whose  proof may trace back its main idea to \cite{rohlin1948generalNotMixing} (see also the proof of "First Category Theorem" in page 77 of \cite{halmosBooklectures}), states that the set $\mathcal R$ of $\mu$-preserving, rigid transformations $T\in\text{Aut}([0,1],\mathcal B,\mu)$ is a dense $G_\delta$ set.\\
It follows that $\mathcal W\cap \mathcal R$ is a dense $G_\delta$ set and, therefore, $\mathcal W\cap \mathcal R\neq\emptyset$. Thus, one can find a $T\in\text{Aut}([0,1],\mathcal B,\mu)$ and increasing sequences $(w_k)_{k\in\N}$ and $(r_k)_{k\in\N}$ in $\N$ with the property that for any $A,B\in\mathcal B$,
$$
\lim_{k\rightarrow\infty}\mu(A\cap T^{-w_k}B)=\mu(A)\mu(B)
$$
and 
$$\lim_{k\rightarrow\infty}\mu(A\cap T^{-r_k}B)=\mu(A\cap B).$$
However, it does not follow from the fact that $\mathcal W\cap \mathcal R\neq \emptyset$ (or any of the refinements of this result that could be obtained from \cite{BdJLRArxivrigidity} or \cite{BdJLRPublished}) that one can take, for example, the sequence $(r_k)_{k\in\N}$ to satisfy the equation $r_k=w_k^2$ for each $k\in\N$. More generally, one may wonder for which sequences $\phi_1,...,\phi_\ell$ the following question has a positive answer:
\begin{question}\label{0.Question1}
    Let $\ell\in\N$ be such that $\ell>1$ and let $\phi_1,...,\phi_\ell:\N\rightarrow \Z$ be such that $\lim_{n\rightarrow\infty}|\phi_j(n)|=\infty$ for each $j\in\{1,...,\ell\}$. Is there a $T\in\text{Aut}([0,1],\mathcal B,\mu)$ and an increasing sequence $(n_k)_{k\in\N}$ in $\N$ with the property that for every $A,B\in\mathcal B$,  
    $$\lim_{k\rightarrow\infty}\mu(A\cap T^{-\phi_j(n_k)}B)=\mu(A)\mu(B)$$
if $j$ is even and 
$$\lim_{k\rightarrow\infty}\mu(A\cap T^{-\phi_j(n_k)}B)=\mu(A\cap B)$$
if $j$ is odd? 
\end{question}
The first results dealing with (the "unitary" variants of)  \Cref{0.Question1} can be found in \cite{BKLUltrafilterPoly}, where the authors introduced the concepts of \textit{$N$-rigidity group} and \textit{group of global rigidity} to deal with this type of question. As it is not part of the scope of this introduction to define any of these two notions, we limit ourselves to remark that (a) our definition of a \textit{group of rigidity} (see \eqref{0.Defn:H(U,n_k)} below) was motivated by both of these concepts, (b) while the concept of a group of rigidity deals with regular convergence and a wide class of families of sequences (see the concept of \textit{adequate sequences} below), the concepts of $N$-rigidity group and that of group of global rigidity are much more restricted (they  only deal with convergence along idempotent ultrafilters and families of polynomial sequences), and (c) the counterparts of the concepts of $N$-rigidity group and group of global rigidity in the context of regular convergence and adequate sequences coincide and are equivalent to the notion of group of rigidity (in a way, this is the content of our main technical result, \cref{0.thm:MainResult} below).\\
The next result, \cref{0.thm:BKLConsequence} below, which follows from   \cite[Theorem E]{BKLUltrafilterPoly} and  \cite[Lemma 3.15]{BKLUltrafilterPoly}, provides a partial answer to \Cref{0.Question1} (and some of its potential "polynomial" variants). We remark that  \cref{0.thm:BKLConsequence} epitomizes well one of the types of results that are intrinsically connected with the concept of group of rigidity. Note the "importance" of the groups generated by any of the sequences 
$$\phi_1,...,\phi_\ell.$$
For any $d\in\N$, we let 
$$\mathcal P_{\leq d}=\{\sum_{j=1}^d a_jx^j\,|\,a_1,...,a_d\in \Z\}.$$
\begin{thm}\label{0.thm:BKLConsequence}
    Let $\ell\in\N$, let $\phi_1,...,\phi_\ell\in\Z[x]$ be non-constant polynomials with zero constant term, and let $F\subseteq \{1,...,\ell\}$. Set $d=\max_{1\leq j\leq \ell}\deg(\phi_j)$.
The following statements are equivalent:
\begin{enumerate}[(I)]
\item There exists a finite index subgroup $H$ of $\mathcal P_{\leq d}$ with the property that for each $j\in F$, $\phi_j\in H$ and for each $j\in \{1,...,\ell\}\setminus F$, $\phi_j\not\in H$.
\item There exist a $T\in\text{Aut}([0,1],\mathcal B,\mu)$ and an increasing sequence $(n_k)_{k\in\N}$ in $\N$ such that 
\begin{equation}\label{0.eq:BKLCases}
\lim_{k\rightarrow\infty}\mu(A\cap T^{-\phi_j(n_k)}B)=\begin{cases}
\mu(A\cap B),\,\text{ if }j\in F,\\
\mu(A)\mu(B),\,\text{ if }j\not\in F,
\end{cases}
\end{equation}
for every $j\in\{1,...,\ell\}$ and  every $A,B\in\mathcal B$.
\end{enumerate}
\end{thm}
\begin{rem}
To illustrate the type of results that one can obtain with the help  of \cref{0.thm:BKLConsequence} as well as the complexity involved in answering \Cref{0.Question1}, we now provide a few examples.
\begin{enumerate}[1]
\item (Cf. \cite[Corollary F]{BKLUltrafilterPoly}.) Let $\phi_1,...,\phi_\ell\in\Z[x]$ be non-constant polynomials with zero constant term. Suppose that $\phi_1,...,\phi_\ell$ are linearly independent. By checking condition (I) in \cref{0.thm:BKLConsequence}, one sees that  \eqref{0.eq:BKLCases} holds for any $F\subseteq \{1,...,\ell\}$.
\item Let $\phi_1(n)=n$ and $\phi_2(n)=2n$ for each $n\in\N$. It is not hard to see that   \eqref{0.eq:BKLCases} holds when $F=\{2\}$ but can never hold with $F=\{1\}$.
    \item (Cf. \cite[Theorem 6.1]{zelada2023GaussianETDS}.) Let $p_1,p_2,p_3\in\N$ be distinct primes and set 
 $$\phi_1(n)=p_1p_2n,\,\phi_2(n)=p_1p_3n,\text{ and }\phi_3(n)=p_2p_3n,\,n\in\N.$$
   One can check that despite the linear dependence of $\phi_1,\phi_2,
   \phi_3$, \eqref{0.eq:BKLCases} holds for any $F\subseteq\{1,2,3\}$.
\end{enumerate}
\end{rem}
As we will see below (see Sections 4 and 5), the concept of group of rigidity not only enables us to generalize \cref{0.thm:BKLConsequence} to a more general class of sequences  but also  provides  us  with the necessary tools to give a "satisfactory" answer to the following question dealing with refinements of St{\"e}pin's interpolation result \cite[Theorem 2]{stepin1987spectral}. We remark that the next question  arose from our previous work \cite{zelada2023GaussianETDS} and provided the original motivation to undertake the present project.
\begin{question}\label{0.Question2}
Let $\ell\in\N$. For which sequences $\phi_1,...,\phi_\ell:\N\rightarrow\Z$ is it true that 
    for each $\vec\lambda=(\lambda_1,...,\lambda_\ell)\in[0,1]^\ell$  one has that the set 
    \begin{multline}\label{0.DefnG_Lambda}
\mathcal G_{\vec \lambda}(\phi_1,...,\phi_\ell):=\{T\in\text{Aut}([0,1],\mathcal B,\mu)\,|\exists (n_t)_{t\in\N}\text{ in }\N\text{ with }\lim_{t\rightarrow\infty}n_t=\infty\,\forall j\in\{1,...,\ell\},\,\\
\forall A,B\in\mathcal B,
\lim_{t\rightarrow\infty}\mu(A\cap T^{- \phi_j(n_t)}B)=(1-\lambda_j)\mu(A\cap B)+\lambda_j\mu(A)\mu(B)\}
\end{multline}
is a dense $G_\delta$ subset of $\text{Aut}([0,1],\mathcal B,\mu)$? 
\end{question}
The following consequence of our main technical result, \cref{0.thm:MainResult} below, provides a full answer to \Cref{0.Question2} under a natural assumption on the sequences $\phi_1,...,\phi_\ell$ (see \Cref{0.rem:AdequateWithoutLossOfGenerality} below for relevant comments).   The sequences $\phi_1,...,\phi_\ell:\N\rightarrow\Z$ are called \textbf{adequate} if for every $j\in\{1,...,\ell\}$,
\begin{equation}\label{0.eq:DivergenceConditionForAdequate}
\lim_{n\rightarrow\infty}|\phi_j(n)|=\infty
\end{equation}
and for any $a_1,...,a_\ell\in\Z$,
\begin{equation}\label{0.eq:LinearCombinationConditionFroAdequate}
\lim_{n\rightarrow\infty}|\sum_{j=1}^\ell a_j\phi_j(n)|\in\{0,\infty\}.
\end{equation}
\begin{thm}\label{0.thm:GenericInterpolationResult}
    Let $\ell\in\N$ and let $\phi_1,...,\phi_\ell:\N\rightarrow\Z$ be such that for any $a_1,...,a_\ell\in\Z$,
\begin{equation}\label{0.eq:IrreducibilityCondition1}
    \lim_{n\rightarrow\infty}\sum_{j=1}^\ell a_j\phi_j(n)
\end{equation}
    exists as an integer, $\infty$, or $-\infty$.
     The following statements are equivalent:
    \begin{enumerate}[(1)]
    \item The sequences $\phi_1,...\phi_\ell$ are adequate and  for any $(a_1,...,a_\ell)\in \Z^\ell$ 
    with 
    $$\lim_{n\rightarrow\infty}\sum_{j=1}^\ell a_j\phi_j(n)=0,$$ 
    we have that for each $j\in\{1,...,\ell\}$, $|a_j|\neq 1$. 
    \item For each $\vec \lambda=(\lambda_1,...,\lambda_\ell)\in[0,1]^\ell$, the set 
     \begin{multline*}
\mathcal G_{\vec \lambda}(\phi_1,...,\phi_\ell)=\{T\in\text{Aut}([0,1],\mathcal B,\mu)\,|\exists (n_t)_{t\in\N}\text{ in }\N\text{ with }\lim_{t\rightarrow\infty}n_t=\infty\,\forall j\in\{1,...,\ell\},\,\\
\forall A,B\in\mathcal B,
\lim_{t\rightarrow\infty}\mu(A\cap T^{- \phi_j(n_t)}B)=(1-\lambda_j)\mu(A\cap B)+\lambda_j\mu(A)\mu(B)\}
\end{multline*}
    is a dense $G_\delta$ set.  
    \item There exists an increasing sequence $(n_k)_{k\in\N}$ in $\N$ with the property that for each $\vec \lambda=(\lambda_1,...,\lambda_\ell)\in[0,1]^\ell$ there exists a weakly mixing $T_{\vec \lambda}\in \text{Aut}([0,1],\mathcal B,\mu)$ such that for any $j\in\{1,...,\ell\}$ and any  $A,B\in\mathcal B$,
    $$\lim_{k\rightarrow\infty}\mu(A\cap T_{\vec \lambda}^{- \phi_j(n_{k})}B)=(1-\lambda_j)\mu(A\cap B)+\lambda_j\mu(A)\mu(B).$$
    \end{enumerate}
\end{thm}
\begin{rem}\label{0.rem:AdequateWithoutLossOfGenerality}
    Observe that assumption \eqref{0.eq:IrreducibilityCondition1} in the statement of \cref{0.thm:GenericInterpolationResult} can be viewed as an "irreducibility" property of the sequences $\phi_1,...,\phi_\ell$: For arbitrary sequences $\psi_1,...,\psi_\ell$ one can always find an increasing sequence $(n_k)_{k\in\N}$ in $\N$ such that for any $a_1,...,a_\ell\in\Z^\ell$,  
    $$\lim_{k\rightarrow\infty}\sum_{j=1}^\ell a_j\psi_j(n_k)\in\Z\cup\{-\infty,\infty\}.$$
\end{rem}
\begin{rem}
    Suppose that $\phi_1,...,\phi_\ell:\N\rightarrow \Z$ satisfy \eqref{0.eq:IrreducibilityCondition1} and condition (1) in \cref{0.thm:GenericInterpolationResult}. It is worth noting that for any increasing sequence $(m_k)_{k\in\N}$ in $\N$, the sequences $(\phi_j(m_k))_{k\in\N}$, $j\in\{1,...,\ell\}$, also satisfy \eqref{0.eq:IrreducibilityCondition1} and condition (1) in \cref{0.thm:GenericInterpolationResult}. Thus, in a certain way, \cref{0.thm:GenericInterpolationResult} deals not only with the sequence $(\phi_1(n),...,\phi_\ell(n))_{n\in\N}$ but also with every subsequence of the form $(\phi_1(m_k),...,\phi_\ell(m_k))_{k\in\N}$. This observation, together with the complete metrisability of $\text{Aut}([0,1],\mathcal B,\mu)$, are the key ingredients for the proof of (2)$\implies$(3) in \cref{0.thm:GenericInterpolationResult} (see Section 5). 
\end{rem}
In a way, the concept of adequate sequences is a far reaching generalization of the concept of families of polynomials with zero constant term. 
It includes $\ell$-tuples of sequences of "shifted" polynomials like 
$$\phi_1(n)=n+1,\phi_2(n)=n^2+2,...,\phi_\ell(n)=n^\ell+\ell$$
and somewhat more general families of sequences such as
$$\phi_1(n)=\lfloor n\alpha_1\rfloor,\phi_2(n)=\lfloor n\alpha_2\rfloor,...,\phi_\ell(n)=\lfloor n\alpha_\ell\rfloor,$$
where $\alpha_1,..,\alpha_\ell$ are non-zero, rationally independent real numbers and for any $r\in\R$, $\lfloor r\rfloor$ denotes the largest integer smaller than or equal to $r$.
\subsection{Definition and characterizations of a rigidity group}
Before stating our main technical result (\cref{0.thm:MainResult} below), we need to introduce the main concepts utilized in this paper.\\

Let $\ell\in\N$, let $\phi_1,...,\phi_\ell:\N\rightarrow\Z$ be adequate sequences, and let $G$ be a subgroup of $\Z^\ell$.
\begin{itemize}
    \item $G$ is called a \textbf{group of rigidity} for the (adequate) sequences $\phi_1,...,\phi_\ell$ (or, simply, a \textit{$(\phi_1,...,\phi_\ell)$-rigidity group}) if there is a strictly increasing sequence $(n_k)_{k\in\N}$ in $\N$ and a unitary operator $U:\mathcal H\rightarrow\mathcal H$ for which the set  
\begin{equation}\label{0.Defn:H(U,n_k)}
H(U,(n_k)):=\{(a_1,...,a_\ell)\in\Z^\ell\,|\,\forall f\in\mathcal H,\,\lim_{k\rightarrow\infty}\|U^{\sum_{j=1}^\ell a_j\phi_j(n_k)}f-f\|_\mathcal H=0\}
\end{equation}
satisfies $G=H(U,(n_k))$. Note that for every unitary $U$ and every increasing sequence $(n_k)_{k\in\N}$ in $
\N$,  $H(U,(n_k))$ is a subgroup of $\Z^\ell$. 
\item We will let $A_G\subseteq \mathbb T^\ell$ denote the \textbf{annihilator of $G$}. In other words,
$$A_G:=\{(\alpha_1,...,\alpha_\ell)\in\mathbb T^\ell\,|\,\forall (a_1,...,a_\ell)\in G,\,e^{2\pi i\sum_{j=1}^\ell a_j\alpha_j}=1\}.$$
We denote the normalized Haar measure of $A_G$ by $\lambda_G$.
\item We will let 
$$A(\phi_1,...,\phi_\ell):=\{(a_1,...,a_\ell)\in\Z^\ell\,|\,\lim_{n\rightarrow\infty}\sum_{j=1}^\ell a_j\phi_j(n)=0\}.$$
Observe that $A(\phi_1,...,\phi_\ell)$ is always a subgroup of $\Z^\ell$. 
\end{itemize}
\begin{thm}[Cf. Theorem 3.5 in \cite{BKLUltrafilterPoly}.]\label{0.thm:MainResult}
    Let $\ell\in\N$ and let $\phi_1,...,\phi_\ell:\N\rightarrow\Z$ be adequate sequences. The following statements  are equivalent for a subgroup $G$ of $\Z^\ell$:
    \begin{enumerate}[(i)]
        \item $G$ is a group of rigidity for $\phi_1,...,\phi_\ell$.
         \item One has that $A(\phi_1,...,\phi_\ell)\subseteq G$.
        \item There exist a Borel probability measure $\sigma$ on $\mathbb T=[0,1)$ and an increasing sequence $(n_k)_{k\in\N}$ in $\N$ with the property that for any continuous function $f:\mathbb T^\ell\rightarrow \mathbb C$ and any measurable $E\subseteq \mathbb T$,
        \begin{equation}\label{0.eq:IndependentDistribution}
            \lim_{k\rightarrow\infty}\int_\mathbb T \mathbbm 1_E(x)f(\phi_1(n_k)x,...,\phi_\ell(n_k)x)\text{d}\sigma(x)=\sigma(E)\int_{\mathbb T^\ell} f(y_1,...,y_\ell)\text{d}\lambda_G(y_1,...,y_\ell),
        \end{equation}
        \item (Cf. Lemma 3.8  in \cite{BKLUltrafilterPoly}) There exist a Borel probability measure $\sigma$ on $\mathbb T=[0,1)$ and an increasing sequence $(n_k)_{k\in\N}$ in $\N$ with the property that for any continuous function $f:\mathbb T^\ell\rightarrow \mathbb C$,
        \begin{equation}\label{eq:(iv)InMainResult}
            \lim_{k\rightarrow\infty}\int_\mathbb T f(\phi_1(n_k)x,...,\phi_\ell(n_k)x)\text{d}\sigma(x)=\int_{\mathbb T^\ell} f(y_1,...,y_\ell)\text{d}\lambda_G(y_1,...,y_\ell).
        \end{equation}
    \end{enumerate}
\end{thm}
\textit{A priori}, one would believe that  the concept of $(\phi_1,...,\phi_\ell)$-rigidity group (see \eqref{0.Defn:H(U,n_k)} above) has very little to do with the limits of the form
\begin{equation}\label{0.eq:LimitOfElementsOutTheGroup}
\lim_{k\rightarrow\infty}U^{\sum_{j=1}^\ell a_j\phi_j(n_k)}f,\,(a_1,...,a_\ell)\not\in G.
\end{equation}
However, as the following equivalent version of item (iii) in \cref{0.thm:MainResult} demonstrates, this is not the case:
\begin{enumerate}
\item [(iii')] There exist an increasing sequence $(n_k)_{k\in\N}$ in $\N$ and a unitary operator $U:\mathcal H\rightarrow\mathcal H$ such that for every $f\in\mathcal H$ and every $(a_1,...,a_\ell)\in\Z^\ell$,
\begin{equation}\label{0.eq:Mixing-RigidCharacterization}
    \lim_{k\rightarrow\infty}U^{\sum_{j=1}^\ell a_j\phi_j(n_k)}f=\begin{cases}
        f\text{ if }(a_1,...,a_\ell)\in G,\\
        0\text{  if }(a_1,...,a_\ell)\not\in G,
    \end{cases}
\end{equation}
in the weak topology of $\mathcal H$.
\end{enumerate}
As a matter of fact, the link between \cref{0.thm:BKLConsequence} and the concept of group of rigidity can be fully appreciated by noting the resemblance between \eqref{0.eq:Mixing-RigidCharacterization} and the following equivalent form of (II) in \cref{0.thm:BKLConsequence}. We let $$L_0^2(\mu):=\{f\in L^2(\mu)\,|\,\int_{[0,1]}f\text{d}\mu=0\}.$$
\begin{enumerate}
    \item [(II')] There exist a $T\in\text{Aut}([0,1],\mathcal B,\mu)$ and an increasing sequence $(n_k)_{k\in\N}$ in $\N$ such that the unitary operator $U_T:L_0^2(\mu)\rightarrow L_0^2(\mu)$ defined by $U_Tf=f\circ T$ satisfies  
\begin{equation}\label{0.eq:BKLCases'}
\lim_{k\rightarrow\infty}U_T^{\phi_j(n_k)}f=\begin{cases}
f,\,\text{ if }j\in F,\\
0,\,\text{ if }j\not\in F,
\end{cases}
\end{equation}
 for every $f\in L_0^2(\mu)$ and every $j\in\{1,...,\ell\}$ in the weak topology of $L_0^2(\mu)$  (the set  $F\subseteq \{1,...,\ell\}$ and the sequences $\phi_1,...,\phi_\ell$ here are as  in the statement of  \cref{0.thm:BKLConsequence}).
\end{enumerate}
We remark in passing that item (iv) in \cref{0.thm:MainResult} is equivalent to the following statement. 
\begin{enumerate}
\item [(iv')] (Cf. Condition (iii') above) There exist an increasing sequence $(n_k)_{k\in\N}$ in $\N$, a unitary operator $U:\mathcal H\rightarrow\mathcal H$, and a non-zero vector  $f\in\mathcal H$ such that for every $(a_1,...,a_\ell)\in\Z^\ell$,
\begin{equation}\label{0.eq:Mixing-RigidCharacterization'}
    \lim_{k\rightarrow\infty}\langle U^{\sum_{j=1}^\ell a_j\phi_j(n_k)}f,f\rangle=\begin{cases}
        \|f\|^2_{\mathcal H}\text{ if }(a_1,...,a_\ell)\in G,\\
        0\text{  if }(a_1,...,a_\ell)\not\in G.
    \end{cases}
\end{equation}
\end{enumerate}

\subsection{Applications of \cref{0.thm:MainResult}}
We now turn our attention to various applications of \cref{0.thm:MainResult}. These applications will allow us to demonstrate the usefulness of most of the different items of \cref{0.thm:MainResult} (including (iii') above).
\subsubsection{ A consequence of (ii)$\iff$(iii')}
One of the key ingredients in the proof of \cref{0.thm:GenericInterpolationResult} is to obtain an "algebraic" characterization of the 
following equivalent form of Condition C in  \cite[Theorem 1.3]{zelada2023GaussianETDS} (see also \cref{5.Cor:AlgebraicChar} below):
\begin{adjustwidth}{0.5cm}{0.5cm}
    \underline{Condition \hypertarget{ConditionC}{C'}}: 
    There exists an increasing sequence $(n_k)_{k\in\N}$ in $\N$ with the property that for any $\vec \xi=(\xi_1,...,\xi_\ell)\in\{0,1\}^\ell$, there is a unitary operator $U_{\vec \xi}:\mathcal H\rightarrow\mathcal H$ such  that for any $f\in\mathcal H$ and any $j\in\{1,...,\ell\}$,
$$\lim_{k\rightarrow\infty}U_{\vec \xi}^{\phi_j(n_k)}f=\xi_jf,$$
in the weak topology of $\mathcal H$.
\end{adjustwidth}
Employing the equivalence (ii)$\iff$(iii) (or, rather, (ii)$\iff$(iii')) in \cref{0.thm:MainResult}, we obtain the following "algebraic" characterization of Condition C'.
\begin{cor}\label{0.CharactMeasurableSingleCorr}
Let $\ell\in\N$ and let $\phi_1,...,\phi_\ell:\N\rightarrow\Z$ be adequate sequences. The sequences $\phi_1,...,\phi_\ell$ satisfy condition C' if and only if for any $(a_1,...,a_\ell)\in A(\phi_1,...,\phi_\ell)$ and any $j\in\{1,...,\ell\}$, $|a_j|\neq 1$.
\end{cor}
The proof of \cref{0.CharactMeasurableSingleCorr} is presented in Section 5. For a related result see \cref{4.Lem:UnitaryResult} below.
\subsubsection{An example in ergodic-Ramsey theory derived from (ii)$\implies$(iii)}
We now provide an example in ergodic-Ramsey theory (see \cref{0.cor:IndependentPolyFailure} below) which is somewhat related with the density polynomial Hales-Jewett conjecture \cite[p. 56]{ERTaU} (see also \cite{GowersDPHJjBlog}) and Question 2 in page 78 of \cite{berMcCuIPPolySzemeredi}. This example  will be constructed with the help of (ii)$\implies$(iii) in \cref{0.thm:MainResult} (see  Subsection 6.2 below). \cref{0.cor:IndependentPolyFailure} 
deals with the notion of IP$^*$,  a notion that plays a prominent role in IP-ergodic theory (see \cite{WMPet}, \cite{berMcCuIPPolySzemeredi},\cite{BFM},\cite{FBook},\cite{FKIPSzemerediLong}, \cite{McCuQuasMildMixing}, \cite{ZorinIPNilpotentSz} for example). A set $E\subseteq \Z$ is called \rm{IP$^*$} if for any strictly monotone sequence $(n_k)_{k\in\N}$ in $\Z$ one has 
$$E\cap \{n_{k_1}+\cdots+n_{k_t}\,|\,k_1<\cdots<k_t,\,t\in\N\}\neq \emptyset.$$
We remark that \cref{0.cor:IndependentPolyFailure} was originally obtained in joint work with Vitaly Bergelson \cite{BerZel-NiceRecurrence} by utilizing the results in \cite{BKLUltrafilterPoly}. We also remark that \cref{0.cor:IndependentPolyFailure} can be viewed as  a variant of   \cite[Theorem A]{ZelIP0Khintchine2023} dealing with multiple recurrence. 
\begin{cor}\label{0.cor:IndependentPolyFailure}
    Let $\ell\in\N$, $\ell\geq 2$, and let $p_1,...,p_\ell\in\Z[x]$ be non-constant polynomials with zero constant term. Suppose that $p_1,...,p_\ell$ are linearly independent. Then,  there exist an invertible measure preserving system $(X,\mathcal A,\nu,T)$, a set $A\in\mathcal A$ with $\nu(A)>0$, and an $\epsilon>0$  for which the set
\begin{equation}\label{0.eq:LargeReturns}
R_\epsilon^{p_1,...,p_\ell}(A)=\{n\in\Z\,|\,\nu(A\cap T^{-p_1(n)}A\cap\cdots\cap T^{-p_\ell(n)}A)>\nu^{\ell+1}(A)-\epsilon\}
\end{equation}
is not \rm{IP$^*$}.
\end{cor}
It is worth mentioning that \cref{0.cor:IndependentPolyFailure} is  somewhat unexpected. This is because the results in \cite{FraKra2006IndependentPolys} and \cite{AlmostIPBerLeib} confirm that, when $p_1,...,p_\ell\in\Z[x]$ are linearly independent and have zero constant term,  the sets of the form $R_\epsilon^{p_1,...,p_\ell}(A)$ are \textit{massive}. More precisely,  by \cite[Theorem 1.3]{FraKra2006IndependentPolys}, we know that every set of the form $R_\epsilon^{p_1,...,p_\ell}(A)$ is syndetic  (meaning that it has bounded gaps) and   by a generalization of \cite[Theorem 4.2]{AlmostIPBerLeib} to appear in \cite{BerZel-NiceRecurrence}, that every set of the form $R_\epsilon^{p_1,...,p_\ell}(A)$
is  "almost" IP$^*$, denoted A-IP$^*$. (A set $E\subseteq \Z$ is called A-IP$^*$ if there exists a set $F\subseteq \Z$ with 
$$d^*(F):=\limsup_{N-M\rightarrow\infty}\frac{|F\cap \{M+1,...,N\}|}{N-M}=0,$$
and such that $E\cup F$ is IP$^*$.) Furthermore, the results in \cite{BFM} show that every set of the form $R_\epsilon^{p_1}(A)$ is IP$^*$ (and, so, \cref{0.cor:IndependentPolyFailure} does not hold when $\ell=1$). It is worth noting  that every A-IP$^*$ set (and, so, every IP$^*$ set)  is syndetic \cite[equation (2.1)]{AlmostIPBerLeib}. \\ 
Invoking the  "inverse Furstenberg's correspondence principle" (See \cite{SohailRobinVanDerCorputSets},\cite{SaulRMInverseFurstenberg}  and \cite{AvigadInverseFurstenberg},\cite{FishInverseFurstenbergCorrespondence}  for related results), we obtain the following "combinatorial" interpretation of these observations.  
\begin{prop}\label{0.prop:CombinatorialExample}
    Let $\ell\in\N$, $\ell\geq 2$, and let $p_1,...,p_\ell\in\Z[x]$ be non-constant polynomials with zero constant term. Suppose that $p_1,...,p_\ell$ are linearly independent. There exists an $\epsilon>0$ with the property that for any F{\o}lner sequence $(F_N)_{N\in\N}$ in $\Z$ there is a set $E\subseteq \Z$ and an $a\in (0,1)$ with 
    $$\overline d_{(F_N)}(E):=\limsup_{N\rightarrow\infty}\frac{|E\cap F_N|}{|F_N|}=a$$
    for which  the set 
    $$\{n\in\Z\,|\,\overline d_{(F_N)}\Big(E\cap (E-p_1(n))\cap \cdots\cap (E-p_\ell(n))\Big)>(
    \overline d_{(F_N)}(E))^{\ell+1}-\epsilon\}$$
    is \rm{A-IP$^*$} but not \rm{IP$^*$}.
\end{prop}
\begin{question}
    Whether or not one can replace $\overline d_{(F_N)}$ by $d^*$ in \cref{0.prop:CombinatorialExample} is an open question. 
\end{question}
\subsubsection{Further examples in ergodic-Ramsey Theory}
The techniques employed to prove \cref{0.cor:IndependentPolyFailure} can be refined to obtain the following two examples (\cref{0.cor:Example2} and \cref{0.cor:Example3} below) each of which can be viewed as an IP-variant of the next result due to V. Bergelson, B. Host, and B. Kra.
\begin{thm}[Theorem 2.1 in \cite{BHKNilSystems2005}]
    There exists an invertible measure preserving system $(X,\mathcal A,\nu,T)$ with the property that for each $\ell\in\N$ there is an $A\in\mathcal A$ with $\nu(A)>0$ such that for every $n\in\Z\setminus\{0\}$,
    $$\nu(A\cap T^{-n}A\cap T^{-2n}A)\leq \frac{\nu^\ell(A)}{2}.$$
\end{thm}
The proof of our first example, \cref{0.cor:Example2} below,  will make use of rigidity groups with infinite index and, so, one could say that it  makes use of the \textit{full strength} of \cref{0.thm:MainResult} (see Subsection 6.3). 
\begin{cor}\label{0.cor:Example2}
    Let $p,q\in\Z[x]$ be non-constant polynomials with zero constant term and such that 
    \begin{equation}\label{0.eq:ConditionOnpAndq}
    \deg(p)=\deg(q)>\deg(2p-q)>0.
    \end{equation}
    Then, there exists an invertible measure preserving system $(X,\mathcal A,\nu,T)$ with the property that for each $\ell\in\N$ there is  an $A\in\mathcal A$ with $\nu(A)>0$ for which the set
    \begin{equation}\label{0.eq:Ex2}
    \{n\in\Z\,|\,\nu(A\cap T^{-p(n)}A\cap T^{-q(n)}A)>\nu^\ell(A)\}
    \end{equation}
    is not \rm{IP$^*$}.
\end{cor}
Our second example, \cref{0.cor:Example3} below,  illustrates that item (iii) in \cref{0.thm:MainResult} can also be employed to obtain information about families of linearly dependent polynomials. (It may be worth noting that formula \eqref{0.eq:Example3} below resembles equation (7) in \cite{DonosoLeSunMoreiraLowerBounds}, which deals with an open question arising from the work of N. Frantzikinakis \cite[Proposition 5.2]{fra2008ThreePolynomials}.)
\begin{cor}\label{0.cor:Example3}
     There exists an invertible measure preserving system $(X,\mathcal A
    ,\nu,T)$ with the property that for each $\ell\in\N$ there is an $A\in\mathcal A$ with $\nu(A)>0$ and such that the set
    \begin{equation}\label{0.eq:Example3}
    \{n\in\Z\,|\,\nu(A\cap T^{-n}A\cap T^{-2n}A\cap T^{-n^2}A)>\nu^\ell(A)\}
    \end{equation}
    is not \rm{IP$^*$}.
\end{cor}
\begin{rem}
 We remark that by employing a technique similar to the one used to prove \cref{0.cor:Example3} (see Subsection 6.4), one can obtain the following strengthened version of \cref{0.cor:Example2}, whose proof we will provide in \cite{BerZel-NiceRecurrence}:\\
 \begin{adjustwidth}{2em}{2em}
     For any $t>1$ and any linearly independent polynomials $p_1,...,p_t\in\Z[x]$ with zero constant term, one can find an invertible measure preserving system $(X,\mathcal A,\nu,T)$ with the property that for every $\ell\in\N$ there is an $A\in\mathcal A$ with $\nu(A)>0$ for which the set 
     $$\{n\in\Z\,|\,\nu(A\cap T^{-p_1(n)}A\cap \cdots\cap T^{-p_t(n)}A)>\nu^\ell(A)\}$$
     is not \rm{IP$^*$}.\\
 \end{adjustwidth}
\end{rem}
The structure of this paper is as follows: In Section 2 we prove a special case of (ii)$\implies$(iii) in  \cref{0.thm:MainResult} which is needed for the proof of \cref{0.thm:MainResult} in its full generality. In Section 3, we  complete the proof of \cref{0.thm:MainResult} and provide additional characterizations of the concept of a group of rigidity. In Section 4 we present a generalized version of \cref{0.thm:BKLConsequence} dealing with adequate sequences. In Section 5 we prove \cref{0.thm:GenericInterpolationResult} and \cref{0.CharactMeasurableSingleCorr}. In Section 6 we prove Corollaries  \ref{0.cor:IndependentPolyFailure}, \ref{0.cor:Example2}, and \ref{0.cor:Example3}.\\

\textit{Acknowledgments.} The author would like to thank Vitaly Bergelson and Joel Moreira  for many helpful discussions. Rigoberto Zelada is supported by EPSRC through Joel Moreira's Frontier Research Guarantee grant, ref. EP/Y014030/1.

\section{A special case of (ii)$\implies$(iii) in  \cref{0.thm:MainResult}}
The sequences $\phi_1,...,\phi_\ell:\N\rightarrow\Z$ are called \textbf{asymptotically linearly independent} if for any $(a_1,...,a_\ell)\in\Z^\ell\setminus\{\vec 0\}$ one has
$$\lim_{n\rightarrow\infty}|\sum_{j=1}^\ell a_j\phi_j(n)|=\infty.$$
In this section we prove a restricted version of (ii)$\implies$(iii) in  \cref{0.thm:MainResult} dealing only with asymptotically linearly independent sequences. We remark that this variant will be used in the next section to prove \cref{0.thm:MainResult} in its full generality and that, because $\phi_1,...,\phi_\ell$ are asymptotically linearly independent, (ii) holds for any subgroup $G$ of $\Z^\ell$: one always has that 
$$\{\vec 0\}=A(\phi_1,...,\phi_\ell)\subseteq G.$$
\begin{lem}\label{1.PropBochnerObservation}
    Let $\ell\in\N$, let $\phi_1,...,\phi_\ell:\N\rightarrow\Z$ be asymptotically linearly independent sequences, and let $G$ be a subgroup of $\Z^\ell$. Then there exist a Borel probability measure $\sigma$ on $\mathbb T=[0,1)$ and an increasing sequence $(n_k)_{k\in\N}$ in $\N$ with the property that for any $a_1,...,a_\ell\in\Z$  and any measurable $E\subseteq \mathbb T$,
        \begin{equation}\label{1.eq:IndependentDistribution}
\lim_{k\rightarrow\infty}\int_{\mathbb T} \mathbbm 1_E(x)e^{2\pi i\left(\sum_{j=1}^\ell a_j\phi_j(n_k)\right)x}\text{d}\sigma(x)=\sigma(E)\int_{\mathbb T^\ell} e^{2\pi i\left(\sum_{j=1}^\ell a_jy_j\right)}\text{d}\lambda_G(y_1,...,y_\ell).
    \end{equation}
\end{lem}
The proof of \cref{1.PropBochnerObservation} requires the use of the following technical lemma, which was proved in \cite{zelada2023GaussianETDS}. Given any  sequences $\phi_1,...,\phi_\ell:\N\rightarrow\Z$, we say that the sequences $\phi_1,...,\phi_\ell$ are \textbf{strongly asymptotically independent} if for any $\vec a=(a_1,...,a_\ell)\in \Z^\ell\setminus\{\vec 0\}$, the sequence 
$$a_1\phi_1(k)+\cdots+a_\ell\phi_\ell(k),\,k\in\N$$
is eventually a strictly monotone sequence. We remark that any strongly asymptotically independent sequences are asymptotically linearly independent.
\begin{lem}[Cf. Theorem 21 in \cite{weyl1916Mod1}]\label{1.WeylsGeneralization}
Let $\phi_1,...,\phi_\ell:\N\rightarrow\Z$ be strongly asymptotically independent sequences. For any $t\in\N$, the set
\begin{multline*}
\mathfrak M_t(\phi_1,...,\phi_\ell)=\\
\{(\alpha_1,...,\alpha_t)\in \R^t\,|\,(\phi_1(k)\alpha_1,...,\phi_\ell(k)\alpha_1,...,\phi_1(k)\alpha_t,...,\phi_\ell(k)\alpha_t)_{k\in\N}\,\text{ is u.d. mod}\,1\}
\end{multline*}
has full Lebesgue measure on $\R^t$. Furthermore, for any $(\alpha_1,...,\alpha_t)\in \mathfrak M_t(\phi_1,...,\phi_\ell)$, the set 
\begin{equation}\label{1.KeyWeylsTypeResult} 
\mathfrak M(\phi_1,...,\phi_\ell,\alpha_1,...,\alpha_t)=\{\alpha\in\R\,|\,(\alpha_1,...,\alpha_t,\alpha)\in \mathfrak M_{t+1}(\phi_1,...,\phi_\ell)\}
\end{equation}
has full measure on $\R$.
\end{lem}
\subsection{The proof of \cref{1.PropBochnerObservation}}
For any $r\in\R$, we let $\|r\|=\inf_{n\in\Z}|r-n|$.
\begin{proof}[Proof of \cref{1.PropBochnerObservation}]
The proof of \cref{1.PropBochnerObservation} is similar to that of \cite[Theorem 4.2]{zelada2023GaussianETDS}.\\

\qedsymbol\textit{ Reduction to proper subgroups of $\Z^\ell$.} We claim that \cref{1.PropBochnerObservation} holds when $G=\Z^\ell$. Indeed,  let $\sigma$ be the Borel probability measure on $\mathbb T$ defined by $\sigma(\{0\})=1$ and let $(n_k)_{k\in\N}$ in $\N$ be an arbitrary increasing sequence. Clearly we have 
$$\lim_{k\rightarrow \infty}\int_\mathbb T \mathbbm 1_E(x)e^{2\pi i(\sum_{j=1}^\ell a_j\phi_j\big(n_k\big))x}\text{d}\sigma(x)=\sigma(E)$$
for every measurable $E\subseteq \mathbb T$ and every $a_1,...,a_\ell\in\Z$, as claimed.\\
\qedsymbol \textit{ Strategy of the proof.} Suppose now that $G$ is a proper subgroup of $\Z^\ell$. We will construct the probability measure $\sigma$ and the increasing sequence $(n_k)_{k\in\N}$ in the following way: First, we utilize \cref{1.WeylsGeneralization} to  find the increasing  sequence  $(n_k)_{k\in\N}$ in $\N$  and  sequences of irrational numbers $(\alpha^{(j)}_k)_{k\in\N}$, $j\in\{1,...,\ell\}$, satisfying various \textit{ad-hoc} Diophantine properties. Then,  we use the sequences $(\alpha^{(j)}_k)_{k\in\N}$, $j\in\{1,...,\ell\}$, to define a continuous function $\Psi$ from the measurable space $(\Omega,\mathcal S)$, where 
$$\Omega=\prod_{k\in\N}\{0,...,k!-1\}^\ell$$
and $\mathcal S=\text{Borel}(\Omega)$, to $\mathbb T$. The probability measure $\sigma$ will be defined by $\sigma=\mathbb P\circ \Psi^{-1}$, where $\mathbb P$ is an appropriately picked  probability measure on $\Omega$ depending on $G$. Finally, we will show that the sequence $(n_k)_{k\in\N}$ and the probability measure $\sigma$ satisfy \eqref{1.eq:IndependentDistribution}.
\subsubsection{ Step 1: Defining $(n_k)_{k\in\N}$ and the sequences $(\alpha^{(j)}_k)_{k\in\N}$, $j\in\{1,...,\ell\}$}
\qedsymbol \textit{ The key properties of $(n_k)_{k\in\N}$ and $(\alpha^{(j)}_k)_{k\in\N}$, $j\in\{1,...,\ell\}$.} For each $k\in\N$ let
$$\Phi(k)=\max_{j\in\{1,...,\ell\}}|\phi_j(k)|+1.$$
 We claim that we can pick the sequence  $(n_k)_{k\in\N}$ in $\N$  and the sequences of irrational numbers $(\alpha^{(j)}_k)_{k\in\N}$, $j\in\{1,...,\ell\}$, to satisfy the following properties:
\begin{enumerate}
    \item [($\sigma$.1)] For each $k\in\N$ and each $j\in\{1,...,\ell\}$, $\alpha^{(j)}_{k}\in(0,\frac{1}{k!2^{k}\Phi(n_{k-1})}]$, where $n_0=1$. So, in particular, for any $j,j'\in\{1,...,\ell\}$,
    $$\lim_{t\rightarrow\infty}|\phi_j(n_t)|\sum_{s=t+1}^\infty (|\alpha^{(j')}_s|\cdot s!)=0.$$
    \item [($\sigma$.2)] For each $k\in\N$ and each $j\in\{1,...,\ell\}$,
    $$\|\phi_j(n_k)\alpha^{(j)}_k-(\frac{1}{k!}+\frac{1}{2(k!)^2})\|<\frac{1}{2 (k!)^2},$$
    which implies that for each $r\in\{0,...,k!-1\}$, $(r\phi_j(n_k)\alpha^{(j)}_k\mod 1)\in [\frac{r}{k!},\frac{r+1}{k!})$ and 
    $$\lim_{t\rightarrow\infty}\|\phi_j(n_t)\alpha^{(j)}_t-(\frac{1}{ t!}+\frac{1}{2(t!)^2})\|=0.$$
    \item [($\sigma$.3)] For each $k\in\N$ and any distinct $j,j'\in\{1,...,\ell\}$,
    $$\|\phi_j(n_k)\alpha^{(j')}_k\|<\frac{1}{2 (k!)^2},$$
    which implies that for each $r\in\{0,...,k!-1\}$, $(r\phi_j(n_k)\alpha^{(j')}_k\mod 1)\in (-\frac{1}{2\cdot k!},\frac{1}{2\cdot k!})$ and 
    $$\lim_{t\rightarrow\infty}\|\phi_j(n_t)\alpha^{(j')}_t\|=0.$$
    \item [($\sigma$.4)] For each $k\in\N$, each $j,j'\in\{1,...,\ell\}$, and each $k_0\in \N$ with $k_0<k$,
    $$\|\phi_j(n_k)\alpha^{(j')}_{k_0}\|<\frac{1}{k^2\cdot k!}.$$
    This means that
    $$\lim_{k\rightarrow\infty}\|\phi_j(n_k)\alpha^{(j')}_{k_0}\|=0$$
    fast enough to ensure that
    $$\lim_{k\rightarrow\infty}\sum_{t=1}^k(\|\phi_j(n_{k+1})\alpha^{(j')}_t\|\cdot k!)=0.$$
\end{enumerate}
\qedsymbol \textit{ Proof of the existence of the sequences $(n_k)_{k\in\N}$ and $(\alpha^{(j)}_k)_{k\in\N}$, $j\in\{1,...,\ell\}$.} We will pick the sequences $(n_k)_{k\in\N}$ and $(\alpha_k^{(j)})_{k\in\N}$, $j
\in\{1,...,\ell\}$, inductively on $k\in\N$. To do this, first note that there exists an increasing  sequence  $(m_k)_{k\in\N}$ in $\N$ for which the sequences 
\begin{equation}\label{2.eq:DefnPsi_j}
\psi_j(k)=\phi_j(m_k),\,j\in\{1,...,\ell\},
\end{equation}
are strongly asymptotically independent. To see this, we first note that because $\Z^\ell$ is countable, the result would follow by a diagonalization argument if we can show that for every non-zero $(a_1,...,a_\ell)\in \Z^\ell$ and any increasing sequence $(r_k)_{k\in\N}$ in $\N$, one can find a subsequence $(s_k)_{k\in\N}$ of $(r_k)_{k\in\N}$ with the property that  $(\sum_{j=1}^\ell a_j\phi_j(s_k))_{k\in\N}$ is a strictly monotone sequence. In turn, one can find such a subsequence by first noting that one can always find a 
 subsequence $(s_k)_{k\in\N}$ of $(r_k)_{k\in\N}$ with the property that either (1) $\sum_{j=1}^\ell a_j\phi_j(s_k)\geq 0$ for each $k\in\N$ or (2) $\sum_{j=1}^\ell a_j\phi_j(s_k)\leq 0$ for every $k\in\N$. 
Since   $\phi_1,...,\phi_\ell$ are asymptotically linearly independent and  for any sequence $(\phi(k))_{k\in\N}$ in $[0,\infty)$ (or $(-\infty,0]$)  with $\lim_{k\rightarrow\infty}|\phi(k)|=\infty$, one can find an increasing sequence $(t_k)_{k\in\N}$ in $\N$ for which  $(\phi(t_k))_{k\in\N}$ is strictly monotone, the result follows.

In order to construct the desired sequences, we will need to show that the sequences $(\alpha^{(j)}_k)_{k\in\N}$, $j\in\{1,...,\ell\}$ satisfy the following additional property:
\begin{enumerate}
\item [($\sigma$.5)] For any $k\in\N$, the sequence 
\begin{multline*}
\big(\phi_1(m_t)\alpha^{(1)}_1,....,\phi_\ell(m_t)\alpha^{(1)}_1,...,\phi_1(m_t)\alpha^{(\ell)}_1,....,\phi_\ell(m_t)\alpha^{(\ell)}_1,\\
\vdots\\
    \phi_1(m_t)\alpha^{(1)}_k,....,\phi_\ell(m_t)\alpha^{(1)}_k,...,\phi_1(m_t)\alpha^{(\ell)}_k,....,\phi_\ell(m_t)\alpha^{(\ell)}_k\big),\,t\in\N
\end{multline*}
is uniformly distributed $\mod\,1$.
\end{enumerate}
For the base case of the induction, note that, by \cref{1.WeylsGeneralization}, we can pick $\alpha_1^{(1)},...,\alpha_1^{(\ell)}\in (0,\frac{1}{2\Phi(1)}]$
such that  the sequence 
    $$\big(\phi_1(m_t)\alpha^{(1)}_1,....,\phi_\ell(m_t)\alpha^{(1)}_1,...,\phi_1(m_t)\alpha^{(\ell)}_1,....,\phi_\ell(m_t)\alpha^{(\ell)}_1\big),\,t\in\N$$
is uniformly distributed $\mod 1$ (and so, $\alpha^{(1)}_1,...,\alpha^{(\ell)}_1$ satisfy ($\sigma$.5)). Pick $t_1\in\N$ arbitrarily and set $n_1=m_{t_1}$. Noting that all of  $\alpha_1^{(1)},...,\alpha_1^{(\ell)}$ must be irrational, we see that  $n_1$ satisfies conditions ($\sigma$.2) and ($\sigma$.3) when $k=1$. Noting that condition ($\sigma$.4) is vacuous for $k=1$, we see that $\alpha^{(1)}_1,...,\alpha^{(\ell)}_1$ and $n_1$ satisfy conditions ($\sigma$.1)-($\sigma$.5).\\ 
Fix now $k\in\N$ and  suppose we have chosen 
$$\alpha^{(1)}_{1},...,\alpha^{(\ell)}_1,...,\alpha^{(1)}_k,...,\alpha^{(\ell)}_k$$ 
and $n_1<\cdots<n_k$ satisfying conditions ($\sigma$.1)-($\sigma$.5). Note that $(0,\frac{1}{(k+1)!2^{k+1}\Phi(n_{k})}]$ has positive measure. By  repeatedly  applying  \eqref{1.KeyWeylsTypeResult} in  \cref{1.WeylsGeneralization}, we can find  $$\alpha_{k+1}^{(1)},...,\alpha_{k+1}^{(\ell)}\in (0,\frac{1}{(k+1)!2^{k+1}\Phi(n_{k})}]$$
such that for each $s\in\{1,...,\ell\}$ the sequence 
\begin{multline*}
\big(\phi_1(m_t)\alpha_1^{(1)},....,\phi_\ell(m_t)\alpha_1^{(1)},...,
    \phi_1(m_t)\alpha_{1}^{(\ell)},...,,\phi_\ell(m_t)\alpha_{1}^{(\ell)},\\
    \vdots\\
    \phi_1(m_t)\alpha_k^{(1)},....,\phi_\ell(m_t)\alpha_k^{(1)},...,
    \phi_1(m_t)\alpha_{k}^{(\ell)},...,,\phi_\ell(m_t)\alpha_{k}^{(\ell)},\\
    \phi_1(m_t)\alpha_{k+1}^{(1)},...,\phi_\ell(m_t)\alpha_{k+1}^{(1)},...,\phi_1(m_t)\alpha_{k+1}^{(s)},...,\phi_\ell(m_t)\alpha_{k+1}^{(s)}
    \big),\,t\in\N
\end{multline*}
is uniformly distributed $\mod 1$. It follows that $\alpha_1^{(j)},...,\alpha_{k+1}^{(j)}$, $j\in\{1,...,\ell\}$, satisfy ($\sigma$.5) and, hence, one can find $t_{k+1}\in\N$ for which ($\sigma$.1)-($\sigma$.4) hold with $n_{k+1}=m_{t_{k+1}}$. Since without loss of generality we can assume that $n_k<n_{k+1}$, we complete the induction.
\subsubsection{ Step 2: Defining $\Psi$ and $\sigma$.}
\qedsymbol \textit{ Defining the function $\Psi$.} Let $\pi$ denote the canonical map from $\R$ to $\mathbb T=\R/\Z$ (so, for each $r\in\R$, $\pi(r)=r\mod 1$ and $\pi$ is continuous). We identify each $\omega\in\Omega$ with the $\Z^\ell$-valued sequence $\omega(k)=(\omega_1(k),...,\omega_\ell(k))$, $k\in\N$, where $\omega_1(k),....,\omega_\ell(k)\in\{0,...,k!-1\}$. Define $g:\Omega\rightarrow \R$  by 
\begin{equation}
g(\omega)=\sum_{s=1}^\infty \left(\sum_{r=1}^\ell\omega_r(s)\alpha_s^{(r)}\right).
\end{equation}
and set $\Psi=\pi\circ g$.\\
\qedsymbol \textit{ The continuity of $\Psi$.} We now prove that $\Psi$ is continuous. To do this, we will show that  $g$ is continuous. Indeed, for each $s\in\N$ and each $r\in\{1,...,\ell\}$, the function $g_{s,r}:\Omega\rightarrow\R$ given by $g_{s,r}(\omega)=\omega_r(s)\alpha_s^{(r)}$ is continuous. Also, by ($\sigma$.1) above, 
for any $\omega\in\Omega$, $|\omega_r(s)\alpha_s^{(r)}|\leq\frac{1}{2^s}$. Thus, by Weierstrass M-test,  we obtain that for each $r\in\{1,...,\ell\}$, the function $g_r:\Omega\rightarrow\R$ defined by $g_r(\omega)=\sum_{s=1}^\infty g_{s,r}(\omega)$  is well defined and continuous. It follows that 
$$g=\sum_{s=1}^\infty\left(\sum_{r=1}^\ell g_{s,r} \right)=\sum_{r=1}^\ell g_r$$
is continuous on $\Omega$.\\
\qedsymbol \textit{ Defining $\sigma$.} Following the notation settled in the Introduction, we  let 
$$A_G:=\{(\alpha_1,...,\alpha_\ell)\in \mathbb T^\ell\,|\,\forall (a_1,...,a_\ell)\in G,
e^{2\pi i\sum_{j=1}^\ell a_j\alpha_j}=1\}$$
and let $\lambda_G$ denote the normalized Haar measure of $A_G$. Observe  that $\lambda_{G}$ can be viewed as a Borel probability measure on $\mathbb T^\ell$. So, in particular, $\lambda_{G}$ is characterized by the quantities
$$\lambda_{G}\left(\prod_{j=1}^\ell[\frac{r_j}{k!},\frac{r_j+1}{k!})\right),$$
where $k\in\N$ and $r_1,...,r_\ell\in\{0,...,k!-1\}$.\\
For each $k\in\N$, let $\mathbb P_k$ be the probability measure on $\{0,...,k!-1\}^\ell$ defined by 
\begin{equation}\label{2.eq:DefnP_k}
\mathbb P_k(\{(r_1,..,r_\ell)\})=\lambda_{G}\left(\prod_{j=1}^\ell[\frac{r_j}{k!},\frac{r_j+1}{k!})\right)
\end{equation}
and set $\mathbb P=\prod_{k=1}^\infty\mathbb P_k$. As we mentioned above, we define $\sigma$ by 
$$\sigma=\mathbb P\circ \Psi^{-1}.$$
\subsubsection{ Step 3: Equation \eqref{1.eq:IndependentDistribution} holds.}
We start by proving that the following Diophantine property holds uniformly on $\omega\in\Omega$ for any $(a_1,...,a_\ell)\in\Z^\ell$:
\begin{equation}\label{1.ClaimOfUniformity}
\lim_{k\rightarrow\infty}\|\left(\sum_{j=1}^\ell a_j\phi_j(n_k)\right)g(\omega)-\sum_{j=1}^\ell a_j\frac{\omega_j(k)}{k!}\|=0.
\end{equation}
\qedsymbol \textit{ Formula \eqref{1.ClaimOfUniformity} holds.} 
Let $(a_1,....,a_\ell)\in\Z^\ell$. Note that for each $k\in\N$, 
$$\|\left(\sum_{j=1}^\ell a_j\phi_j(n_k)\right)g(\omega)-\sum_{j=1}^\ell a_j\frac{\omega_j(k)}{k!}\|\leq \sum_{j=1}^\ell|a_j|\left\|\phi_j(n_k)g(\omega)-\frac{\omega_j(k)}{k!}\right\|.$$
Furthermore, for each $j\in\{1,...,\ell\}$, 
$$\|\phi_j(n_k)g(\omega)-\frac{\omega_j(k)}{k!}\|=\|\phi_j(n_k)\sum_{r=1}^\ell g_r(\omega)-\frac{\omega_j(k)}{k!}\|\leq \sum_{r\neq j}\|\phi_j(n_k)g_r(\omega)\|+\|\phi_j(n_k)g_j(\omega)-\frac{\omega_j(k)}{k!}\|.$$
Thus, in order to prove \eqref{1.ClaimOfUniformity}, all we need to show is that for any $r,j\in\{1,...,\ell\}$,
\begin{equation}\label{1.ClaimOfUniformity2}
\lim_{k\rightarrow\infty}\|\phi_j(n_k)g_r(\omega)-\delta_j(r)\frac{\omega_r(k)}{k!}\|=0
\end{equation}
uniformly in $\omega\in\Omega$.\\
Fix $j,r\in\{1,...,\ell\}$. By ($\sigma$.1) above and the fact that each of the sequences $\psi_1,...,\psi_\ell$ defined in \eqref{2.eq:DefnPsi_j} are eventually monotone, we have that for any $\omega\in \Omega$,
\begin{equation*}
    \limsup_{k\rightarrow\infty}\|\phi_j(n_k)\sum_{s=k+1}^\infty (\omega_r(s) \alpha_s^{(r)})\|\leq \limsup_{k\rightarrow\infty} |\phi_j(n_k)|\sum_{s=k+1}^\infty (|\alpha_s^{(r)}|\cdot s!)=0.
\end{equation*}
Thus, 
\begin{equation}\label{1.ClaimOfUniformity3.1}
     \lim_{k\rightarrow\infty}\|\phi_j(n_k)\sum_{s=k+1}^\infty (\omega_r(s) \alpha_s^{(r)})\|=0
\end{equation}
uniformly in $\omega\in\Omega$.\\
By ($\sigma$.4) above, for any $\omega\in\Omega$,
$$\limsup_{k\rightarrow\infty}\|\phi_j(n_k)\sum_{s=1}^{k-1}(\omega_r(s)\alpha_s^{(r)})\|\leq 
\limsup_{k\rightarrow\infty}\sum_{s=1}^{k-1}(\|\phi_j(n_{k})\alpha^{(r)}_s\|\cdot s!)=0.$$
Thus,
\begin{equation}\label{1.ClaimOfUniformity3.2}
    \lim_{k\rightarrow\infty}\|\phi_j(n_k)\sum_{s=1}^{k-1}(\omega_r(s)\alpha_s^{(r)})\|=0 
\end{equation}
uniformly in $\omega\in\Omega$.\\
By ($\sigma$.2) and ($\sigma$.3), for any $\omega\in\Omega$, 
\begin{multline*}
\limsup_{k\rightarrow\infty}\|\phi_j(n_k)\omega_r(k)\alpha_k^{(r)}-\delta_j(r)\frac{\omega_r(k)}{k!}\|\leq \limsup_{k\rightarrow\infty}|\omega_r(k)|\left\|\phi_j(n_k)\alpha_k^{(r)}-\frac{\delta_j(r)}{k!} \right\|\\
\leq \limsup_{k\rightarrow\infty}\frac{k!}{(k!)^2}=0.
\end{multline*}
Thus,
\begin{equation}\label{1.ClaimOfUniformity3.3}
\lim_{k\rightarrow\infty}\|\phi_j(n_k)\omega_r(k)\alpha_k^{(r)}-\delta_j(r)\frac{\omega_r(k)}{k!}\|=0
\end{equation}
uniformly in $\omega\in\Omega$. 
Combining \eqref{1.ClaimOfUniformity3.1}, \eqref{1.ClaimOfUniformity3.2}, and \eqref{1.ClaimOfUniformity3.3}, we see that \eqref{1.ClaimOfUniformity2} holds.\\

\qedsymbol \textit{ Proof of \eqref{1.eq:IndependentDistribution}.} Observe that in order to show that \eqref{1.eq:IndependentDistribution} holds, all we need to show is that for any $(a_1,...,a_\ell)\in\Z^\ell$ and any $m\in\Z$ one has,
$$
\lim_{k\rightarrow\infty}\int_{\mathbb T} e^{2\pi imx}e^{2\pi i\left(\sum_{j=1}^\ell a_j\phi_j(n_k)\right)x}\text{d}\sigma(x)=\int_\mathbb T e^{2\pi imx}\text{d}\sigma(x)\int_{\mathbb T^\ell} e^{2\pi i\left(\sum_{j=1}^\ell a_jy_j\right)}\text{d}\lambda_G(y_1,...,y_\ell).
$$
(This follows because  trigonometric polynomials are dense in $L^1(\sigma)$.)\\
To prove this, fix $(a_1,...,a_\ell)\in\Z^\ell$ and let $m\in\Z$. By  the definitions of $\sigma$ and $\mathbb P$, we have
\begingroup
\allowdisplaybreaks
\begin{multline}\label{2.3.eq1}
    \lim_{k\rightarrow\infty}\int_\mathbb T e^{2\pi i[\left(\sum_{j=1}^\ell a_j\phi_j(n_k)\right)+m]x}\text{d}\sigma(x)=\lim_{k\rightarrow\infty}\int_\Omega e^{2\pi i[\left(\sum_{j=1}^\ell a_j\phi_j(n_k)\right)+m]\Psi(\omega)}\text{d}\mathbb P(\omega)\\
    =\lim_{k\rightarrow\infty}\int_\Omega e^{2\pi i\left(\sum_{j=1}^\ell a_j\phi_j(n_k)\right)\Psi(\omega)}e^{2\pi im\Psi(\omega)}\text{d}\mathbb P(\omega)\\
    =\lim_{k\rightarrow\infty}\int_\Omega e^{2\pi i\left(\sum_{j=1}^\ell a_j\phi_j(n_k)\right)g(\omega)}e^{2\pi img(\omega)}\text{d}\mathbb P(\omega)
\end{multline}
Applying first \eqref{1.ClaimOfUniformity} to the last expression in \eqref{2.3.eq1} and then replacing $g$ with one of its approximations, we obtain 
\begin{multline}\label{2.3eq2}
    \lim_{k\rightarrow\infty}\int_\mathbb T e^{2\pi i[\left(\sum_{j=1}^\ell a_j\phi_j(n_k)\right)+m]x}\text{d}\sigma(x)\\
    =\lim_{k\rightarrow\infty}\int_\Omega e^{2\pi i\sum_{j=1}^\ell a_j\frac{\omega_j(k)}{k!}}e^{2\pi im\left[\sum_{s=1}^{k-1}\sum_{j=1}^\ell \omega_j(s)\alpha_s^{(j)}\right]}\text{d}\mathbb P(\omega)
\end{multline}
Since $\mathbb P$ is a product measure, the last  expression in \eqref{2.3eq2} equals 
    \begin{multline*}
    \lim_{k\rightarrow\infty}\int_\Omega e^{2\pi i\sum_{j=1}^\ell a_j\frac{\omega_j(k)}{k!}}\text{d}\mathbb P(\omega)\int_\Omega e^{2\pi im\left[\sum_{s=1}^{k-1}\sum_{j=1}^\ell\omega_j(s)\alpha_s^{(j)}\right]}\text{d}\mathbb P(\omega)\\
    =\lim_{k\rightarrow\infty}\int_\Omega e^{2\pi i\sum_{j=1}^\ell a_j\frac{\omega_j(k)}{k!}}\text{d}\mathbb P(\omega)\lim_{k\rightarrow\infty}\int_\Omega e^{2\pi im\left[\sum_{s=1}^{k-1}\sum_{j=1}^\ell \omega_j(s)\alpha_s^{(j)}\right]}\text{d}\mathbb P(\omega)\\
    =\left(\lim_{k\rightarrow\infty}\int_\Omega e^{2\pi i\sum_{j=1}^\ell a_j\frac{\omega_j(k)}{k!}}\text{d}\mathbb P(\omega)\right)\int_\Omega e^{2\pi im\Psi(\omega)}\text{d}\mathbb P(\omega)\\
    =\int_\mathbb T e^{2\pi imx}\text{d}\sigma(x)\left(\lim_{k\rightarrow\infty}\int_\Omega e^{2\pi i\sum_{j=1}^\ell a_j\frac{\omega_j(k)}{k!}}\text{d}\mathbb P(\omega)\right).
    \end{multline*}
    Finally, by \eqref{2.eq:DefnP_k},
    \begin{multline*}
    \lim_{k\rightarrow\infty}\int_\mathbb T e^{2\pi i[\left(\sum_{j=1}^\ell a_j\phi_j(n_k)\right)+m]x}\text{d}\sigma(x)\\
    =\int_\mathbb T e^{2\pi imx}\text{d}\sigma(x)\left(\lim_{k\rightarrow\infty}\int_{\{0,...,k!-1\}^\ell}e^{2\pi i\sum_{j=1}^\ell a_j\frac{r_j}{k!}}\text{d}\mathbb P_k(r_1,...,r_\ell)\right)\\
    =\int_\mathbb T e^{2\pi imx}\text{d}\sigma(x)\left(\lim_{k\rightarrow\infty} \sum_{r_1,...,r_\ell=0}^{k!-1}\left[e^{2\pi i\sum_{j=1}^\ell a_j\frac{r_j}{k!}}\cdot\lambda_G\left(\prod_{j=1}^\ell[\frac{r_j}{k!},\frac{r_j+1}{k!})\right)\right]\right)\\
    =\int_\mathbb T e^{2\pi imx}\text{d}\sigma(x)\int_{\mathbb T^\ell}e^{2\pi i[\sum_{j=1}^\ell a_jy_j]}\text{d}\lambda_G(y_1,...,y_\ell).
\end{multline*}
\endgroup
We are done.
\end{proof}
\section{The proof of our main technical result}
In this section we prove \cref{3.thm:MainResult}, a refinement of \cref{0.thm:MainResult} in the Introduction. Observe that item (v) in \cref{3.thm:MainResult} below is not mentioned in \cref{0.thm:MainResult}. As demonstrated below, item (v)  serves as an intermediate step in the proof of (ii)$\implies$(iii). We remark that the numerals in \cref{3.thm:MainResult} and \cref{3.cor:MoreCharacterizations} below coincide with the numerals used in \cref{0.thm:MainResult} and formulas \eqref{0.eq:Mixing-RigidCharacterization} and \eqref{0.eq:Mixing-RigidCharacterization'} in the Introduction.
\begin{thm}\label{3.thm:MainResult}
    Let $\ell\in\N$ and let $\phi_1,...,\phi_\ell:\N\rightarrow\Z$ be adequate sequences. The following statements  are equivalent for a subgroup $G$ of $\Z^\ell$:
    \begin{enumerate}[(i)]
        \item $G$ is a group of rigidity for $\phi_1,...,\phi_\ell$.
         \item One has that $A(\phi_1,...,\phi_\ell)\subseteq G$.
        \item There exist a Borel probability measure $\sigma$ on $\mathbb T=[0,1)$ and an increasing sequence $(n_k)_{k\in\N}$ in $\N$ with the property that for any continuous function $f:\mathbb T^\ell\rightarrow \mathbb C$ and any measurable $E\subseteq \mathbb T$,
        \begin{equation}\label{3.eq:IndependentDistribution}
            \lim_{k\rightarrow\infty}\int_\mathbb T \mathbbm 1_E(x)f(\phi_1(n_k)x,...,\phi_\ell(n_k)x)\text{d}\sigma(x)=\sigma(E)\int_{\mathbb T^\ell} f(y_1,...,y_\ell)\text{d}\lambda_G(y_1,...,y_\ell),
        \end{equation}
        \item There exist a Borel probability measure $\sigma$ on $\mathbb T=[0,1)$ and an increasing sequence $(n_k)_{k\in\N}$ in $\N$ with the property that for any continuous function $f:\mathbb T^\ell\rightarrow \mathbb C$,
        \begin{equation}
            \lim_{k\rightarrow\infty}\int_\mathbb T f(\phi_1(n_k)x,...,\phi_\ell(n_k)x)\text{d}\sigma(x)=\int_{\mathbb T^\ell} f(y_1,...,y_\ell)\text{d}\lambda_G(y_1,...,y_\ell).
        \end{equation}
         \item There exist $c\in\{1,...,\ell\}$, $M\in\N$,  $j_1,...,j_c\in\{1,...,\ell\}$ with $j_1<\cdots<j_c$, and a subgroup $\tilde G$ of $\Z^c$ such that (a) the sequences $\phi_{j_1},...,\phi_{j_c}$ are asymptotically linearly independent, (b) there is a homomorphism $\theta:\Z^\ell\rightarrow \Z^c$ defined by   $(a_1,...,a_c)=\theta(d_1,...,d_\ell)$ if and only if 
\begin{equation}\label{3.SurjectiveHiddenMap1}
\lim_{k\rightarrow\infty}|\sum_{r=1}^ca_r\phi_{j_r}(k)-M\sum_{j=1}^\ell d_j\phi_j(k)|=0,
\end{equation}
and (c) one has that 
$G=\theta^{-1}(\tilde G)
$.
    \end{enumerate}
\end{thm}
Before embarking in the proof of \cref{3.thm:MainResult}, we remark that the probability measure $\lambda_G$ on items (iii) and (iv) above is the only Borel probability measure $\rho$ on $\mathbb T^\ell$ with the property that for any $(a_1,...,a_\ell)\in \Z^\ell$, 
\begin{equation}\label{3.eq:CharHaarMeasureONA_g}
\int_{\mathbb T^\ell} e^{2\pi i \sum_{j=1}^\ell a_jx_j}\text{d}\rho(x_1,...,x_\ell)=\begin{cases}
    1,\text{ if }(a_1,...,a_\ell)\in G,\\
    0,\text{ if }(a_1,...,a_\ell)\not\in G.
\end{cases}
\end{equation}
Indeed, it suffices to show that 
\begin{equation*}\label{3.AnnihilatorOfAnnihilator}
B_G:=\{(a_1,...,a_\ell)\in \Z^\ell\,|\,\forall (\alpha_1,...,\alpha_\ell)\in A_{G},\,e^{2\pi i\sum_{j=1}^\ell a_j\alpha_j}=1\}.
\end{equation*}
satisfies $G=B_G$. To check tat $G=B_G$, 
let $\pi$ denote the canonical homomorphism from $\Z^\ell$ to $\Z^\ell/G$. Observe that a homomorphism $\sigma:\Z^\ell\rightarrow\mathbb S^1$ belongs to $A_G$ if and only if there is a homomorphism $\psi:\Z^\ell/G\rightarrow \mathbb S^1$ such that $\sigma=\psi\circ \pi$. Thus, noting that for every $(a_1,...,a_\ell)\in\Z^\ell\setminus G$, there is a homomorphism $\psi:\Z^\ell/G\rightarrow\mathbb S^1$ with $\psi((a_1,...,a_\ell)+G)\neq 1$, we obtain that $(\Z^\ell\setminus G)\cap B_G=\emptyset$. Since $G\subseteq B_G$ follows from the definition of $A_G$, we get $G=B_G$.
(That $G=B_G$ can also be viewed as an immediate consequence of   \cite[Lemma 2.1.3]{rudin1962fourier}, for example.)
\begin{proof}[Proof of \cref{3.thm:MainResult}]
(i)$\implies$(ii): Suppose that  $G$ is a $(\phi_1,...,\phi_\ell)$-rigidity group. Then, there is a unitary operator  $U$ and an increasing sequence $(n_k)_{k\in\N}$ in $\N$ for which the group $H(U,(n_k))$ (as defined in \eqref{0.Defn:H(U,n_k)}) satisfies $G=H(U,(n_k))$. Since every $(a_1,...,a_\ell)\in A(\phi_1,...,\phi_\ell)$ satisfies $(a_1,...,a_\ell)\in H(U,(n_k))$, we see that (ii) holds.\\

(ii)$\implies$(v): We remark that conditions (a) and (b) hold for any adequate sequences $\phi_1,...,\phi_\ell$. Condition (ii) is only needed for the proof of (c).\\
\qedsymbol \textit{ Proof of (a).} By \eqref{0.eq:DivergenceConditionForAdequate}, one has that when $c=1$, one can find $j_1,...,j_c\in\{1,...,\ell\}$, $j_1<\cdots<j_c$, with the property that  for any non-zero $(a_1,...,a_c)\in\Z^c$,
$$
\lim_{n\rightarrow \infty}|\sum_{r=1}^ca_r\phi_{j_r}(n)|=\infty.
$$
Taking the largest $c\in\{1,...,\ell\}$ for which this property holds, one obtains that 
$\{\phi_{j_1},...,\phi_{j_c}\}\subseteq\{\phi_1,...,\phi_\ell\}$ is a maximal subset of
asymptotically linearly independent sequences. Note that by \eqref{0.eq:LinearCombinationConditionFroAdequate}, the maximality of $\{\phi_{j_1},...,\phi_{j_c}\}$ implies that   for any  $j\in\{1,...,\ell\}\setminus\{j_1,...,j_c\}$, there exist $(b_{j_1}^{(j)},...,b_{j_c}^{(j)})\in\Z^c\setminus\{\vec 0\}$ and a $b_j\in\N$ such that $\lim_{n\rightarrow\infty}|\sum_{r=1}^cb_{j_r}^{(j)}\phi_{j_r}(n)-b_j\phi_j(n)|=0$.
Thus,  for any  $j\in\{1,...,\ell\}$, there exist $(b_{j_1}^{(j)},...,b_{j_c}^{(j)})\in\Z^c\setminus\{\vec 0\}$ and a $b_j\in\N$ such that 
\begin{equation}\label{4.LinearCombinationLemma}
\lim_{n\rightarrow\infty}|\sum_{r=1}^cb_{j_r}^{(j)}\phi_{j_r}(n)-b_j\phi_j(n)|=0.
\end{equation}
\qedsymbol \textit{ Preparations for the proofs of (b) and (c).} By re-indexing $\phi_1,...,\phi_\ell$, if needed, we can assume without loss of generality that $j_r=r$ for each $r\in\{1,...,c\}$. In addition to this, we can also assume without loss of generality that for each $j\in\{1,...,\ell\}$, the integers  $b_1^{(j)},...,b_c^{(j)},b_j$ in \eqref{4.LinearCombinationLemma} satisfy $\text{GCD}(b_1^{(j)},...,b_c^{(j)},b_j)=1$. \\
Let $M=\prod_{j=1}^\ell b_j$ and let 
$$\tilde G=\{(a_1,...,a_c)\in\Z^c\,|\,\exists (d_1,...,d_\ell)\in G,\, \lim_{n\rightarrow\infty}|\sum_{r=1}^ca_r\phi_r(n)-M\sum_{j=1}^\ell d_j\phi_j(n)|=0\}.$$
\qedsymbol \textit{ Proof of (b).} Define the map $\theta:\Z^\ell\rightarrow \Z^c$ as in the statement of (v) (with $j_r=r$ for each $r\in\{1,...,c\})$. Note that if $\theta$ is a well-defined function, it is routine to check that $\theta$ is a  homomorphism. In turn, by the  asymptotic linear independence of $\phi_1,...,\phi_c$, in order to check that $\theta$ is a function all that one needs to check is that  for any $d_1,...,d_\ell\in \Z$ there exists an $(a_1,...,a_c)\in \Z^c$ with 
\begin{equation}\label{2.SurjectiveHiddenMap}
\lim_{n\rightarrow\infty}|\sum_{r=1}^ca_r\phi_r(n)-M\sum_{j=1}^\ell d_j\phi_j(n)|=0.
\end{equation}
Finally, to check \eqref{2.SurjectiveHiddenMap},  define $a_r:=\sum_{j=1}^\ell\left(d_j\frac{M}{b_j}b_r^{(j)}\right)$ for each $r\in\{1,...,c\}$. We have 
\begin{multline*}
    \lim_{n\rightarrow\infty}|\sum_{r=1}^ca_r\phi_r(n)-M\sum_{j=1}^\ell d _j\phi_j(n)|=\lim_{n\rightarrow\infty}|\sum_{r=1}^c\sum_{j=1}^\ell \left(d_j\frac{M}{b_j}b_r^{(j)}\right)\phi_r(n)-M\sum_{j=1}^\ell d_j\phi_j(n)|\\
    =\lim_{n\rightarrow\infty}|\sum_{j=1}^\ell d_j\frac{M}{b_j}\left(\sum_{r=1}^cb_r^{(j)}\phi_r(n)\right)-M\sum_{j=1}^\ell d_j\phi_j(n)|=\lim_{n\rightarrow\infty}|M\sum_{j=1}^\ell \frac{d_j}{b_j}\cdot(b_j\phi_j(n))-M\sum_{j=1}^\ell d_j\phi_j(n)|=0,
\end{multline*}
proving that (b) holds.\\ 
\qedsymbol \textit{ Proof of (c).} Note that $\text{Ker}(\theta)=A(\phi_1,...,\phi_\ell)$ and that $\theta(G)=\tilde G$. Thus, by (ii), $\theta^{-1}(\tilde G)=G$.\\

(v)$\implies$(iii): \qedsymbol \textit{ Consequences of (v) and the definitions of $(n_k)_{k\in\N}$ and $\sigma$.} Let $c\in\N$, $M\in\N$, and $\tilde G\subseteq \Z^c$ be as in the statement of (v). By re-indexing $\phi_1,...,\phi_\ell$, if needed, we assume without loss of generality that $j_r=r$ for each $r\in\{1,...,c\}$. By \cref{1.PropBochnerObservation} and (a) in (v), there exists an increasing sequence $(n_k)_{k\in\N}$ in $\N$ and a Borel probability measure $\rho$ on $\mathbb T$ with the property that for any $m\in\Z$ and any $a_1,...,a_c\in\Z$,
\begin{equation}\label{3.eq:OptionLimitForGTilde1}
\lim_{k\rightarrow\infty}\int_\mathbb T e^{2\pi i(m+\sum_{j=1}^ca_j\phi_j\big(n_k\big))x}\text{d}\rho(x)=\begin{cases}
\int_\mathbb T e^{2\pi imx}\text{d}\rho(x),\,\text{ if }(a_1,...,a_c)\in \tilde G,\\
0,\,\text{ if }(a_1,...,a_c)\not\in \tilde G. 
\end{cases}
\end{equation}
Let $\psi_M:\mathbb T\rightarrow\mathbb T$ be the map defined by $\psi_M(x)=Mx\mod 1$. We set $\sigma=\rho\circ \psi_M^{-1}$.\\
\qedsymbol \textit{ Consequences of \eqref{3.eq:OptionLimitForGTilde1} for the measure $\sigma$.} Let $d_1,...,d_\ell\in\Z$ and let $m\in\Z$. Suppose that $(d_1,...,d_\ell)\in G$. Then, by (b) and (c) in item (v), there exists $(a_1,...,a_c)\in \tilde G$ for which \eqref{3.SurjectiveHiddenMap1} holds. So, by \eqref{3.eq:OptionLimitForGTilde1},
\begin{multline*}
\lim_{k\rightarrow\infty}\int_\mathbb T e^{2\pi i(m+\sum_{j=1}^\ell d_j\phi_j\big(n_k\big))x}\text{d}\sigma(x)=\lim_{k\rightarrow\infty}\int_\mathbb T e^{2\pi iM(m+\sum_{j=1}^\ell d_j\phi_j\big(n_k\big))x}\text{d}\rho(x)\\
=\lim_{k\rightarrow\infty}\int_\mathbb T e^{2\pi iMmx}e^{2\pi i\sum_{r=1}^c a_r\phi_r(n_k)x}\text{d}\rho(x)=\int_\mathbb T e^{2\pi iMmx}\text{d}\rho(x)=\int_\mathbb T e^{2\pi imx}\text{d}\sigma(x),
\end{multline*}
When $(d_1,...,d_\ell)\not\in G$, (c) in (v) implies that there is an $(a_1,...,a_c)\in\Z^c\setminus \tilde G$ for which \eqref{3.SurjectiveHiddenMap1} holds. Thus, 
\begin{multline*}
\lim_{k\rightarrow\infty}\int_\mathbb T e^{2\pi i(m+\sum_{j=1}^\ell d_j\phi_j\big(n_k\big))x}\text{d}\sigma(x)=\lim_{k\rightarrow\infty}\int_\mathbb T e^{2\pi iM(m+\sum_{j=1}^\ell d_j\phi_j\big(n_k\big))x}\text{d}\rho(x)\\
=\lim_{k\rightarrow\infty}\int_\mathbb T e^{2\pi iMmx}e^{2\pi i\sum_{r=1}^c a_r\phi_r(n_k)x}\text{d}\rho(x)=0.
\end{multline*}
\qedsymbol \textit{ Deduction of \eqref{3.eq:IndependentDistribution}.} It follows that for any $(d_1,...,d_\ell)\in\Z^\ell$ and any $m\in\Z$,
\begin{equation*}\label{3.eq:OptionLimitForGTilde}
\lim_{k\rightarrow\infty}\int_\mathbb T e^{2\pi i(m+\sum_{j=1}^\ell d_j\phi_j\big(n_k\big))x}\text{d}\sigma(x)=\begin{cases}
\int_\mathbb T e^{2\pi imx}\text{d}\sigma(x),\,\text{ if }(d_1,...,d_\ell)\in  G,\\
0,\,\text{ if }(d_1,...,d_\ell)\not\in G.
\end{cases}
\end{equation*}
So, in light of \eqref{3.eq:CharHaarMeasureONA_g}, for any $(d_1,...,d_\ell)\in\Z^\ell$ and any $m\in\Z$,
$$
\lim_{k\rightarrow\infty} \int_\mathbb T e^{2\pi i(m+\sum_{j=1}^\ell d_j\phi_j\big(n_k\big))x}\text{d}\sigma(x)=\int_\mathbb T e^{2\pi imx}\text{d}\sigma(x)\int_{\mathbb T^\ell} e^{2\pi i\sum_{j=1}^\ell d_jy_j}\text{d}\lambda_G(y_1,...,y_\ell),
$$
which, in virtue of the density of the set of trigonometric polynomials on $L^1(\sigma)$, proves \eqref{3.eq:IndependentDistribution}.\\

(iii)$\implies$(iv): This implication is trivial.\\

(iv)$\implies$(i): We will prove that $G=H(U,(n_k))$ for some unitary operator $U$ and the sequence $(n_k)_{k\in\N}$ in (iv).\\ 
\qedsymbol \textit{ Definition and basic properties of $U$.} Let $\sigma$ be as in (iv) and let  $ U:L^2(\sigma)\rightarrow L^2(\sigma)$ be defined by $Uf(\cdot)=e^{2\pi i(\cdot)}f(\cdot)$. By (iv), we have that for any $(a_1,...,a_\ell)\in G$, 
$$\lim_{k\rightarrow\infty} e^{2\pi i(\sum_{j=1}^\ell a_j\phi_j\big(n_k\big))x}=1$$
in $\sigma$-measure.\\
\qedsymbol \textit{ $G\subseteq H(U,(n_k))$.} It follows that for any $f\in L^2(\sigma)$ and any $(a_1,...,a_\ell)\in G$,
$$\lim_{k\rightarrow\infty}U^{\sum_{j=1}^\ell a_j\phi_j(n_k)}f=f$$
in $L^2$-norm and, so, $G\subseteq H(U,(n_k))$. \\
\qedsymbol \textit{ $H(U,(n_k))\subseteq G$.} To see that $H(U,(n_k))\subseteq G$, we will prove that if $(a_1,...,a_\ell)\not\in G$, then
$$\lim_{k\rightarrow\infty}U^{\sum_{j=1}^\ell a_j\phi_j(n_k)}\mathbbm 1_\mathbb T\neq \mathbbm 1_\mathbb T.$$
Indeed,  
\begin{multline*}
\lim_{k\rightarrow\infty}\langle U^{\sum_{j=1}^\ell a_j\phi_j(n_k)}\mathbbm 1_{\mathbb T},\mathbbm 1_\mathbb T\rangle =\lim_{k\rightarrow\infty}\int_\mathbb T e^{2\pi i\sum_{j=1}^\ell a_j\phi_j(n_k)x}\text{d}\sigma(x)\\
=\int_{\mathbb T^\ell}e^{2\pi i\sum_{j=1}^\ell a_jy_j}\text{d}\lambda_G(y_1,...,y_\ell)=0\neq 1=\langle \mathbbm 1_\mathbb T,\mathbbm 1_\mathbb T\rangle.
\end{multline*}
We are done. 
\end{proof}
We complete this section by recording several additional characterizations of a $(\phi_1,...,\phi_\ell)$-rigidity group. We remark that these characterizations can be deduced from  \cref{3.thm:MainResult} by invoking Bochner's Theorem (see \cite[Section 1.4]{rudin1962fourier}, for example) and the fact that  for any unitary operator $U:\mathcal H\rightarrow\mathcal H$ and any non-zero $f\in\mathcal H$, the sequence $(\langle U^nf,f\rangle)_{n\in\Z}$ is positive definite. We omit the proofs.
\begin{cor}\label{3.cor:MoreCharacterizations}
    Let $\ell\in\N$ and let $\phi_1,...,\phi_\ell:\N\rightarrow\Z$ be adequate sequences. The following statements  are equivalent for a subgroup $G$ of $\Z^\ell$:
    \begin{enumerate}
        \item [(i)] $G$ is a group of rigidity for $\phi_1,...,\phi_\ell$.
        \item [(iii')]  There exist an increasing sequence $(n_k)_{k\in\N}$ in $\N$ and a unitary operator $U:\mathcal H\rightarrow\mathcal H$ such that for every $f\in\mathcal H$ and every $(a_1,...,a_\ell)\in\Z^\ell$,
\begin{equation*}
    \lim_{k\rightarrow\infty}U^{\sum_{j=1}^\ell a_j\phi_j(n_k)}f=\begin{cases}
        f\text{ if }(a_1,...,a_\ell)\in G,\\
        0\text{  if }(a_1,...,a_\ell)\not\in G,
    \end{cases}
\end{equation*}
in the weak topology of $\mathcal H$
\item [(iv')] There exist an increasing sequence $(n_k)_{k\in\N}$ in $\N$, a unitary operator $U:\mathcal H\rightarrow\mathcal H$, and a non-zero vector  $f\in\mathcal H$ such that for every $(a_1,...,a_\ell)\in\Z^\ell$,
\begin{equation}
    \lim_{k\rightarrow\infty}\langle U^{\sum_{j=1}^\ell a_j\phi_j(n_k)}f,f\rangle=\begin{cases}
        \|f\|^2_{\mathcal H}\text{ if }(a_1,...,a_\ell)\in G,\\
        0\text{  if }(a_1,...,a_\ell)\not\in G.
    \end{cases}
\end{equation}
\end{enumerate}
\end{cor}
\section{A generalization of \cref{0.thm:BKLConsequence}}
Our goal in this section is  to obtain the following generalized version of \cref{0.thm:BKLConsequence}. For any $\ell\in\N$, we denote by $\vec e_1,...,\vec e_\ell$ the standard basis of $\Z^\ell$ (so, in particular, $\vec e_j$ is the $j$-th row/column of the $\ell\times \ell$ identity matrix).
\begin{thm}\label{4.thm:BKLConsequenceGeneralized}
    Let $\ell\in\N$, let $\phi_1,...,\phi_\ell:\N\rightarrow\Z$ be adequate sequences, and let $F\subseteq \{1,...,\ell\}$.
The following statements are equivalent:
\begin{enumerate}[(I)]
\item There exists a finite index subgroup $H$ of $\Z^\ell$ with $A(\phi_1,...,\phi_\ell)\subseteq H$ and such that for any $j\in\{1,...,\ell\}$,  $j\in F$ if and only if $\vec e_j\in H$.
\item There exist a $T\in\text{Aut}([0,1],\mathcal B,\mu)$ and an increasing sequence $(n_k)_{k\in\N}$ in $\N$ such that 
\begin{equation}\label{0.eq:BKLCasesG}
\lim_{k\rightarrow\infty}\mu(A\cap T^{-\phi_j(n_k)}B)=\begin{cases}
\mu(A\cap B),\,\text{ if }j\in F,\\
\mu(A)\mu(B),\,\text{ if }j\not\in F,
\end{cases}
\end{equation}
for every $j\in\{1,...,\ell\}$ and  every $A,B\in\mathcal B$.
\item For any $(a_1,...,a_\ell)\in A(\phi_1,...,\phi_\ell)$, either 
$$|\{i\in\{1,...,\ell\}\,|\,a_i\neq 0\}\setminus F|\neq 1$$
or the only $j$ in $\{i\in\{1,...,\ell\}\,|\,a_i\neq 0\}\setminus F$ satisfies  $|a_j|\neq 1$.
\end{enumerate}
\end{thm}
The proof of \cref{4.thm:BKLConsequenceGeneralized}  relies in the following closely related result of unitary nature (see  \cref{4.Lem:UnitaryResult} below) which will be  proved in  Subsection 4.1 and several basic facts about Gaussian systems (see Subsection 4.2). We remark that \cref{4.Lem:UnitaryResult} is also needed for the proof of \cref{0.CharactMeasurableSingleCorr} which we present in the next section. 
\begin{lem}\label{4.Lem:UnitaryResult}
 Let $\ell\in\N$, let $\phi_1,...,\phi_\ell:\N\rightarrow\Z$ be adequate sequences, and let $F\subseteq \{1,...,\ell\}$.
The following statements are equivalent:
\begin{enumerate}[(a)]
\item There exists a finite index subgroup $H$ of $\Z^\ell$ with $A(\phi_1,...,\phi_\ell)\subseteq H$ and such that for any $j\in\{1,...,\ell\}$,  $j\in F$ if and only if $\vec e_j\in H$.
\item 
There exist an increasing sequence $(n_k)_{k\in\N}$ in $\N$ and a unitary operator $U:\mathcal H\rightarrow\mathcal H$, such that for every $f\in\mathcal H$ and every $j\in\{1,...,\ell\}$,
\begin{equation*}
    \lim_{k\rightarrow\infty}U^{\phi_j(n_k)}f=\begin{cases}
        f\text{ if }j\in F,\\
        0\text{  if }j\not\in F,
    \end{cases}
\end{equation*}
in the weak topology of $\mathcal H$.
\item For any $(a_1,...,a_\ell)\in A(\phi_1,...,\phi_\ell)$, either 
$$|\{i\in\{1,...,\ell\}\,|\,a_i\neq 0\}\setminus F|\neq 1$$
or the only $j$ in $\{i\in\{1,...,\ell\}\,|\,a_i\neq 0\}\setminus F$ satisfies  $|a_j|\neq 1$.
\end{enumerate}
\end{lem}
\subsection{The proof of \cref{4.Lem:UnitaryResult}}
The proof of \cref{4.Lem:UnitaryResult} makes use of the following result proved in \cite{BKLUltrafilterPoly}.
\begin{lem}[Lemma 3.15 in \cite{BKLUltrafilterPoly}]\label{4.lem:ExtendToFiniteIndex}
Let $\ell,t\in\N$, let $G$ be a subgroup of $\Z^\ell$, and let $\vec a_1,...,\vec a_t\in\Z^\ell$. Suppose that $\vec a_1,...,\vec a_t\not\in G$. Then, there exists a finite index subgroup $H$ of $\Z^\ell$ with $G\subseteq H$ and such that  $\vec a_1,...,\vec a_t\in\Z^\ell\setminus H$.
\end{lem}
\begin{proof}[Proof of \cref{4.Lem:UnitaryResult}]
    \textit{ (c)$\implies$(a)}: By \cref{4.lem:ExtendToFiniteIndex}, it suffices to show that there exists a subgroup $G$ of $\Z^\ell$ such that $A(\phi_1,...,\phi_\ell)\subseteq G$ and for each $j\in \{1,...,\ell\}$, $j\in F$ if and only if  $\vec e_j\in G$.\\
    To do this, set 
    \begin{equation}\label{4.eq:DefinitionOfG=A+G_F}
    G:=\{(a_1,...,a_\ell)+\sum_{i\in F}b_i\vec e_i\,|\,(a_1,...,a_\ell)\in A(\phi_1,...,\phi_\ell)\text{ and }\forall i\in F,\, b_i\in \Z\}.
    \end{equation}
    Clearly, $A(\phi_1,...,\phi_\ell)\subseteq G$ and for each $j\in F$, $\vec e_j\in G$.\\
    Let $j\in\{1,...,\ell\}$. All that remains to be shown is that  if  $\vec e_j\in G$, then $j\in F$. To do this, suppose that $\vec e_j\in G$. It follows that for some $(a_1,...,a_\ell)\in A(\phi_1,...,\phi_\ell)$ and $b_i\in \Z$, $i\in F$,  
    \begin{equation}\label{4.FormulaForE_j}
    \vec e_j=(a_1,...,a_\ell)+\sum_{i\in F}b_i\vec e_i.
    \end{equation}
    In turn, \eqref{4.FormulaForE_j} implies that for each $i\in \{1,...,\ell\}\setminus\{j\}$, $a_i=-b_i$ if $i\in F$ and $a_i=0$ otherwise. Thus, exactly one of $\{i\in \{1,...,\ell\}\,|\,a_i\neq 0\}\setminus F=\emptyset$ or $\{i\in \{1,...,\ell\}\,|\,a_i\neq 0\}\setminus F=\{j\}$ holds. If $\{i\in \{1,...,\ell\}\,|\,a_i\neq 0\}\setminus F=\emptyset$, then $j\in F$ as desired. Hence,  the result would follow if we could show that (since (c) holds), $\{i\in \{1,...,\ell\}\,|\,a_i\neq 0\}\setminus F=\{j\}$ is impossible.\\
    Indeed, suppose for contradiction that $\{i\in \{1,...,\ell\}\,|\,a_i\neq 0\}\setminus F=\{j\}$ holds.  On the one hand, because $j\not\in F$, we must have $a_j=1$. On the other hand, our assumption (c) implies that $|a_j|\neq 1$. We are done.\\
    \textit{ (a)$\implies$(b)}:  The result follows immediately from  \cref{0.thm:MainResult} (and \cref{3.cor:MoreCharacterizations}).\\
    \textit{ (b)$\implies$(c)}: Suppose that (b) holds and let $(a_1,...,a_\ell)\in A(\phi_1,...,\phi_\ell)$. We will show that under the additional assumption that $|\{i\in\{1,...,\ell\}\,|\,a_i\neq 0\}\setminus F|=1$, one must have that the only $j\in\{i\in\{1,...,\ell\}\,|\,a_i\neq 0\}\setminus F$ satisfies $|a_j|\neq 1$. Indeed, let $j$ be the only member of $\{i\in\{1,...,\ell\}\,|\,a_i\neq 0\}\setminus F$ and suppose for contradiction that $|a_j|=1$. Note that by replacing $(a_1,...,a_\ell)$ with $-(a_1,...,a_\ell)$, if needed, we can assume  without loss of generality that $a_j=1$. It follows that for some $b_i\in \Z$, $i\in F$, $\vec e_j+\sum_{i\in F}b_i\vec e_i\in A(\phi_1,...,\phi_\ell)$ and, so, 
    $$\phi_j(n)=-\sum_{i\in F}b_i\phi_i(n),$$
    for $n\in\N$ large enough. Letting $U$ and $(n_k)_{k\in\N}$ be as in the statement of (b), we obtain that for every $f\in\mathcal H$,
    $$\lim_{k\rightarrow\infty}U^{\phi_j(n_k)}f=\lim_{k\rightarrow\infty}U^{-\sum_{i\in F}b_i\phi_i(n_k)}f=f,$$
    which implies that $j\in F$. A contradiction. 
    \end{proof}
\subsection{Background on Gaussian systems}
In this subsection we review the necessary background material on Gaussian systems.\\

Let $\mathcal A=\text{Borel}(\R^\Z)$ and consider the measurable space $(\R^\Z,\mathcal A)$. For each $n\in\Z$, we will let 
\begin{equation}\label{1.ProjectionDefn}
X_n:\R^\Z\rightarrow \R
\end{equation}
denote the projection onto the $n$-th coordinate (i.e. for each $\omega\in\R^\Z$, $X_n(\omega)=\omega(n)$).\\
It is well known that for any unitary operator $U:\mathcal H\rightarrow\mathcal H$ and any $\vec v\in\mathcal H$ with $\|\vec v\|_{\mathcal H}=1$, there exists a unique probability measure $\gamma=\gamma_{U,\vec v}:\mathcal A\rightarrow [0,1]$ such that (a) for any 
$f\in H_1=\overline{\text{span}_{\R}\{X_n\,|\,n\in\Z\}^{L^2(\gamma)}}$, 
$f$ has a Gaussian distribution with mean zero\footnote{We will treat the constant function $f=0$ as a normal random variable with variance zero.} and (b)
for any $m,n\in\Z$,
\begin{equation}\label{1.CorrelationsCoincide}
\int_{\R^\Z} X_nX_m\text{d}\gamma=\mathfrak{Re}(\langle U^{n-m}\vec v ,\vec v\rangle).
\end{equation}
We call the probability measure $\gamma$ the \textbf{Gaussian measure associated with $U$ and $\vec v$}. As we will see below, many of the properties of the sequence $(U^n\vec v)_{n\in\Z}$ (and hence $H_1$) are intrinsically connected with those of $\gamma$.\\

Let $T:\R^\Z\rightarrow\R^\Z$ denote the shift map defined by
$$[T(\omega)](n)=\omega(n+1)$$
for each $\omega\in\R^\Z$ and each $n\in\Z$. The quadruple $(\R^\Z,\mathcal A,\gamma,T)$ is an invertible probability measure preserving system called the \textbf{Gaussian system associated with the pair $(U,\vec v)$}. (For the construction of a Gaussian system, see  \cite[Chapter 8]{cornfeld1982ergodic} or \cite[Appendix C]{kechris2010Global}, for example.)\\    
The following result follows from Theorem 2.3 in \cite{zelada2023GaussianETDS}. We omit the proof. 
\begin{thm}\label{4.thm:limitsOfGaussina}
    Let $U:\mathcal H\rightarrow\mathcal H$ be a unitary operator and let $\vec v\in\mathcal H$ be such that $\|\vec v\|_{\mathcal H}=1$. Let
$(\R^\Z,\mathcal A,\gamma,T)$ be the Gaussian system associated with $(U,\vec v)$. Given a  sequence $\phi:\N\rightarrow\Z$, the following statements hold:
\begin{enumerate}
    \item [(G.1)] Suppose that 
    $$\lim_{n\rightarrow\infty} U^{\phi(n)}\vec v=0$$
    in the weak topology of $\mathcal H$. Then, for any $A,B\in\mathcal A$,
    $$\lim_{n\rightarrow\infty}\gamma(A\cap T^{-\phi(n)}B)=\gamma(A)\gamma(B).$$
    \item [(G.2)] Suppose that  $$\lim_{n\rightarrow\infty} U^{\phi(n)}\vec v=\vec v$$
    in the weak topology of $\mathcal H$. Then, for any $A,B\in\mathcal A$,
    $$\lim_{n\rightarrow\infty}\gamma(A\cap T^{-\phi(n)}B)=\gamma(A\cap B).$$
\end{enumerate}
\end{thm}
\subsection{The proof of \cref{4.thm:BKLConsequenceGeneralized}}
\begin{proof}[Proof of \cref{4.thm:BKLConsequenceGeneralized}] Because items (I) and (III) in \cref{4.thm:BKLConsequenceGeneralized} are identical to items (a) and (c) in \cref{4.Lem:UnitaryResult}, respectively, all that we need to show is that item (II) in \cref{4.thm:BKLConsequenceGeneralized} is equivalent to item (b) in \cref{4.Lem:UnitaryResult}.\\
\textit{(II)$\implies$(b)}: Let $T$ and $(n_k)_{k\in\N}$ be as in the statement of (II). Set 
$$L_0^2(\mu)=\{f\in L^2(\mu)\,|\,\int_{[0,1]}f\text{d}\mu=0\}$$ 
and consider the unitary operator $U_T:L_0^2(\mu)\rightarrow L_0^2(\mu)$ defined by $U_Tf=f\circ T$. Clearly, (b) is satisfied by $U_T$ and the sequence $(n_k)_{k\in\N}$.\\
\textit{(b)$\implies$(II)}: Let $U:\mathcal H\rightarrow\mathcal H$ and $(n_k)_{k\in\N}$ be as in the statement of (b). Pick $\vec v\in\mathcal H$ with $\|\vec v\|_{\mathcal H}=1$ and let $(\R^\Z,\mathcal A,\gamma,T)$ be the Gaussian system associated with $(U,\vec v)$. Because $\gamma\circ X_0^{-1}$ is a non-degenerate  Gaussian measure, we have that $\gamma$ is  non-atomic (i.e. for every $\omega\in\R^\Z$, $\gamma(\{\omega\})=0$). Therefore  $(\R^\Z,\mathcal A,\gamma,T)$ is measure theoretically isomorphic to $([0,1],\mathcal B,\mu,S)$ for some $S\in \text{Aut}([0,1],\mathcal B,\mu)$ (see \cite[Theorem 2.1]{waltersIntroduction}, for example). It now follows from  (G.1) and (G.2) in \cref{4.thm:limitsOfGaussina} that (II) holds.
\end{proof}
\section{The proof of \cref{0.thm:GenericInterpolationResult} and \cref{0.CharactMeasurableSingleCorr}}
In this section we prove \cref{0.thm:GenericInterpolationResult}. In Subsection 5.1 we will prove a refined version of \cref{0.CharactMeasurableSingleCorr}. In Subsection 5.2 we combine this refined version of \cref{0.CharactMeasurableSingleCorr} with one of the results in  \cite{zelada2023GaussianETDS} to prove that  (1)$\implies$(2) in \cref{0.thm:GenericInterpolationResult} holds. In Subsection 5.3, we  refresh some of the topological properties of $\text{Aut}([0,1],\mathcal B,\mu)$ and use them to prove (2)$\implies$(3). Finally, in Subsection 5.4 we complete the proof of \cref{0.thm:GenericInterpolationResult} by showing that (3)$\implies$(1) holds. 
\subsection{ A refined version of \cref{0.CharactMeasurableSingleCorr}}
We now utilize \cref{4.thm:BKLConsequenceGeneralized} and \cref{4.Lem:UnitaryResult} to prove the following result. A transformation $T\in \text{Aut}([0,1],\mathcal B,\mu)$ is called aperiodic if the set of $x\in[0,1]$ for which there exists an $n\in\N$ with $T^nx=x$ has measure zero (so, for example, any weakly mixing $T\in \text{Aut}([0,1],\mathcal B,\mu)$ is aperiodic).
\begin{cor}\label{5.Cor:AlgebraicChar}
    Let $\ell\in\N$ and let $\phi_1,...,\phi_\ell$ be adequate sequences. The following statements are equivalent.
    \begin{enumerate}[(a)]
\item For any $(a_1,...,a_\ell)\in A(\phi_1,...,\phi_\ell)$ and every $j\in\{1,...,\ell\}$, $|a_j|\neq 1$.
\item (Condition \hyperlink{ConditionC}{C'} above.) There exists an increasing sequence $(n_k)_{k\in\N}$ in $\N$  with the property that for any $\vec \xi=(\xi_1,...,\xi_\ell)\in\{0,1\}^\ell$ there is a unitary operator $U_{\vec \xi}:\mathcal H\rightarrow\mathcal H$ such that for any $f\in \mathcal H$ and any  $j\in \{1,...,\ell\}$, 
$$\lim_{k\rightarrow\infty}U_{\vec \xi}^{\phi_j(n_k)}f=\xi_j f$$
in the weak topology of $\mathcal H$.
\item (Condition C in \cite{zelada2023GaussianETDS}.) There exists an increasing sequence  $(n_k)_{k\in\N}$  in $\N$  with the property that for any $\vec \xi=(\xi_1,...,\xi_\ell)\in\{0,1\}^\ell$ there is an aperiodic  $T_{\vec \xi}\in\text{Aut}([0,1],\mathcal B,\mu)$ such  that for any $j\in\{1,...,\ell\}$ and any  $A,B\in\mathcal B$, 
\begin{equation}\label{0.Mixing/RigidCondition}
\lim_{k\rightarrow\infty}\mu(A\cap T_{\vec \xi}^{- \phi_j(n_k)}B)=(1-\xi_j)\mu(A\cap B)+\xi_j\mu(A)\mu(B).
\end{equation}
    \end{enumerate}
\end{cor}
\begin{proof}
    \textit{(a)$\implies$(b)}:  Suppose that condition (a) in \cref{5.Cor:AlgebraicChar} holds and let $ F_1,...,F_{2^\ell}$ be an enumeration of the subsets of $\{1,...,\ell\}$.   Applying (c)$\implies$(b) in 
 \cref{4.Lem:UnitaryResult} with $F=F_1$, we can find an increasing sequence $(n^{(1)}_k)_{k\in\N}$ in $\N$ and a unitary operator $U_1:\mathcal H\rightarrow\mathcal H$, such that for every $f\in\mathcal H$ and every $j\in\{1,...,\ell\}$,
\begin{equation*}
    \lim_{k\rightarrow\infty}U_1^{\phi_j(n_k^{(1)})}f=\begin{cases}
        f\text{ if }j\in F_1,\\
        0\text{  if }j\not\in F_1,
    \end{cases}
\end{equation*}
in the weak topology of $\mathcal H$. Note now that the sequences $\phi^{(1)}_j(k)=\phi_j(n^{(1)}_k)$, $j\in\{1,...,\ell\}$, form a family of adequate sequences for which condition (a) in \cref{5.Cor:AlgebraicChar} holds. Thus, by applying  (c)$\implies$(b) in 
 \cref{4.Lem:UnitaryResult} to the sequences $\phi_1^{(1)}$,...,$\phi_\ell^{(1)}$ and with $F=F_2$, we can find a subsequence $(n_k^{(2)})_{k\in\N}$ of $(n^{(1)}_k)_{k\in\N}$ and a unitary operator $U_2:\mathcal H\rightarrow\mathcal H$, such that for every $f\in\mathcal H$ and every $j\in\{1,...,\ell\}$,
\begin{equation*}
    \lim_{k\rightarrow\infty}U_2^{\phi_j(n_k^{(2)})}f=\begin{cases}
        f\text{ if }j\in F_2,\\
        0\text{  if }j\not\in F_2.
    \end{cases}
\end{equation*}
Continuing in this way, we can find sequences $(n_k^{(1)})_{k\in\N}$,...,$(n_k^{(2^\ell)})_{k\in\N}$ and unitary operators $U_1,...,U_{2^\ell}$ such that for each $i\in\{1,...,2^\ell-1\}$, $(n_k^{(i+1)})_{k\in\N}$ is a subsequence of $(n_k^{(i)})_{k\in\N}$ and for every $i\in\{1,...,2^\ell\}$, every $f\in\mathcal H$, and every $j\in\{1,...,\ell\}$,
\begin{equation*}
    \lim_{k\rightarrow\infty}U_i^{\phi_j(n_k^{(i)})}f=\begin{cases}
        f\text{ if }j\in F_i,\\
        0\text{  if }j\not\in F_i.
    \end{cases}
\end{equation*}
Taking $(n_k)_{k\in\N}=(n_k^{(2^\ell)})_{k\in\N}$, we see that (b) in \cref{5.Cor:AlgebraicChar} holds.\\

\textit{(b)$\implies$(a)}: Let $(a_1,...,a_\ell)\in A(\phi_1,...,\phi_\ell)$ and let $j\in\{1,...,\ell\}$. Observe that the set 
$$F:=\{1,...,\ell\}\setminus\{j\}$$
satisfies 
$$\{i\in\{1,...,\ell\}\,|\,a_i\neq 0\}\setminus F\subseteq \{j\}.$$
Thus, by (b) in  \cref{5.Cor:AlgebraicChar} and (b)$\implies$(c) in \cref{4.Lem:UnitaryResult}, we obtain that $|a_j|\neq 1$.\\

\textit{(c)$\implies$(a)}: The proof is similar to that of (b)$\implies$(a) except that one utilizes (II)$\implies$(III) in \cref{4.thm:BKLConsequenceGeneralized}
in place of (b)$\implies$(c) in \cref{4.Lem:UnitaryResult}.\\

\textit{(b)$\implies$(c)}: Suppose that (b) in \cref{5.Cor:AlgebraicChar} holds and note that, as shown above, we must have that   (a) in \cref{5.Cor:AlgebraicChar} also holds. Let $\phi_{\ell+1}:\N\rightarrow\Z$ be a sequence with the property that for any $a_1,...,a_\ell\in \Z$ and any $a_{\ell+1}\in \N$, $\lim_{n\rightarrow\infty}|\sum_{j=1}^{\ell+1}a_j\phi_j(n)|=\infty$. It follows that $\phi_1,...,\phi_{\ell+1}$ are adequate sequences and that $A(\phi_1,...,\phi_{\ell+1})=A(\phi_1,...,\phi_\ell)\times\{0\}$. So, in particular, $\phi_1,...,\phi_{\ell+1}$ satisfy condition (a) in \cref{5.Cor:AlgebraicChar}. By condition (b) in  \cref{5.Cor:AlgebraicChar} we can find an increasing sequence $(n_k)_{k\in\N}$ in $\N$, such that for every $\vec \eta=(\eta_1,...,\eta_{\ell+1})\in \{0,1\}^{\ell+1}$, there is a unitary operator $U_{\vec \eta}:\mathcal H\rightarrow\mathcal H$ such that for every $f\in\mathcal H$ and every $j\in\{1,...,\ell+1\}$,
\begin{equation*}
    \lim_{k\rightarrow\infty}U_{\vec \eta}^{\phi_j(n_k)}f=(1-\eta_j)f.
\end{equation*}
Let $\vec \eta=(\eta_1,...,\eta_{\ell+1})\in\{0,1\}^{\ell+1}$ be such that $\eta_{\ell+1}=1$. Picking $
\vec v\in\mathcal H$ with $\|\vec v\|_{\mathcal H}=1$, setting $\vec \xi=(\eta_1,...,\eta_\ell)$, and letting $T_{\vec \xi}\in\text{Aut}([0,1],\mathcal B,\mu)$ be measure-theoretically  isomorphic to the Gaussian system associated with $(U_{\vec \eta},\vec v)$, we see that \cref{4.thm:limitsOfGaussina} implies that $T_{\vec \xi}$ satisfies   \eqref{0.Mixing/RigidCondition} and, because $\eta_{\ell+1}=1$,  $T_{\vec \xi}$ is weakly mixing (so, in particular, it is aperiodic).
\end{proof}
\subsection{The proof of (1)$\implies$(2) in \cref{0.thm:GenericInterpolationResult}}
We now prove (1)$\implies$(2) in \cref{0.thm:GenericInterpolationResult} (see \cref{5.Lem:(1)implica(2)} below). 
\begin{lem}\label{5.Lem:(1)implica(2)}
    Let  $\phi_1,...\phi_\ell:\N\rightarrow\Z$ be adequate sequences. Suppose that 
    \begin{equation}\label{5.eq:QuasiIndependenceCondition}
    \text{for every }(a_1,....,a_\ell)\in A(\phi_1,...,\phi_\ell)\text{ and every }j\in\{1,...,\ell\},
    \,|a_j|\neq 1.
    \end{equation}
    Then, for each $\vec \lambda=(\lambda_1,...,\lambda_\ell)\in[0,1]^\ell$, the set 
     \begin{multline*}
\mathcal G_{\vec \lambda}(\phi_1,...,\phi_\ell)=\{T\in\text{Aut}([0,1],\mathcal B,\mu)\,|\exists (n_t)_{t\in\N}\text{ in }\N\text{ with }\lim_{t\rightarrow\infty}n_t=\infty\,\forall j\in\{1,...,\ell\},\,\\
\forall A,B\in\mathcal B,
\lim_{t\rightarrow\infty}\mu(A\cap T^{- \phi_j(n_t)}B)=(1-\lambda_j)\mu(A\cap B)+\lambda_j\mu(A)\mu(B)\}
\end{multline*}
    is a dense $G_\delta$ set.  
\end{lem}
We will prove \cref{5.Lem:(1)implica(2)} with the help of \cref{5.Cor:AlgebraicChar} and \cite[Theorem 1.3]{zelada2023GaussianETDS}, which we now state.
\begin{thm}[Theorem 1.3 in \cite{zelada2023GaussianETDS}]\label{5.Theorem1.3InGaussian}
Let $\ell\in\N$, let $\vec \lambda=(\lambda_1,...,\lambda_\ell)\in[0,1]^\ell$, and let $\phi_1,...,\phi_\ell:\N\rightarrow\Z$. Suppose that $\phi_1,...,\phi_\ell$ satisfy item (c) in \cref{5.Cor:AlgebraicChar}. Then,  the set $\mathcal G_{\vec \lambda}(\phi_1,...,\phi_\ell)$ 
is a dense $G_\delta$ set.
\end{thm}
\begin{proof}[Proof of \cref{5.Lem:(1)implica(2)}]
Let $\phi_1,...,\phi_\ell$ be as in the statement of \cref{5.Lem:(1)implica(2)}. By \cref{5.Cor:AlgebraicChar}, we have that $\phi_1,...,\phi_\ell$ satisfy condition (c) in \cref{5.Cor:AlgebraicChar} and, so, by \cref{5.Theorem1.3InGaussian}, we have that for every $\vec \lambda=(\lambda_1,...,\lambda_\ell)\in [0,1]^\ell$, the set $\mathcal G_{\vec \lambda}(\phi_1,...,\phi_\ell)$ is a dense $G_\delta$ set. 
\end{proof}
\subsection{The proof of (2)$\implies$(3) in \cref{0.thm:GenericInterpolationResult}}
We will now prove that (2)$\implies$(3) in \cref{0.thm:GenericInterpolationResult} (see \cref{5.lem:(2)Implies(3)} below).
\begin{lem}\label{5.lem:(2)Implies(3)}
    Let $\phi_1,...,\phi_\ell:\N\rightarrow \Z$ be such that for any $a_1,...,a_\ell\in\Z$,
    \begin{equation}\label{5.eq:AssumptionThatAllLimitsExsist}
    \lim_{n\rightarrow\infty}\sum_{j=1}^\ell a_j\phi_j(n)\in\Z\cup\{-\infty,\infty\}.
    \end{equation}
    Suppose that  for each  $\vec \lambda\in[0,1]^\ell$, the set $\mathcal G_{\vec \lambda}(\phi_1,...,\phi_\ell)$ as defined in \eqref{0.DefnG_Lambda}
    is a dense $G_\delta$ set.  
    There exists an increasing sequence $(n_k)_{k\in\N}$ in $\N$ with the property that for each $\vec \lambda=(\lambda_1,...,\lambda_\ell)\in[0,1]^\ell$ there exists a weakly mixing $T_{\vec \lambda}\in \text{Aut}([0,1],\mathcal B,\mu)$ such  that for any $A,B\in\mathcal B$ and any $j\in\{1,...,\ell\}$,
    \begin{equation}\label{5.eq:LambdaMixingInLemma}
    \lim_{k\rightarrow\infty}\mu(A\cap T_{\vec \lambda}^{- \phi_j(n_{k})}B)=(1-\lambda_j)\mu(A\cap B)+\lambda_j\mu(A)\mu(B).
    \end{equation}
\end{lem}
Before presenting  the proof of \cref{5.lem:(2)Implies(3)}, we need to refresh several facts about $\text{Aut}([0,1],\mathcal B,\mu)$.\\

For the rest of this section let $(A_k)_{k\in\N}$ be a sequence in $\mathcal B$ with the property that for any $B\in\mathcal B$ and any $\epsilon>0$, there is a $k\in\N$ with $\mu(A_k\triangle B)<\epsilon$, where $A_k\triangle B$ denotes the symmetric difference between $A_k$ and $B$. We define the complete metric $\partial$ on  $\text{Aut}([0,1],\mathcal B,\mu)$ by 
\begin{equation}\label{5.eq:DefnMetric}
    \partial(T,S):=\frac{1}{2}\sum_{r\in\N}\frac{\mu(TA_r\triangle SA_r)}{2^r}+\frac{1}{2}\sum_{r\in\N}\frac{\mu(T
    ^{-1}A_r\triangle S^{-1}A_r)}{2^r}.
\end{equation}
The topology induced by $\partial$ on $\text{Aut}([0,1],\mathcal B,\mu)$ is the so called \textit{weak topology}, alluded in the introduction. As we mentioned in the Introduction, when endowed with the weak topology, $\text{Aut}([0,1],\mathcal B,\mu)$ becomes a topological group (see \cite{kechris2010Global}, for example). The following result due to Halmos \cite[page 77]{halmosBooklectures} will be instrumental in the proof of \cref{5.lem:(2)Implies(3)}.
\begin{lem}[Conjugacy Lemma]\label{5.lem:ConjugacyLem}
Let $T\in\textit{Aut}([0,1],\mathcal B,\mu)$ be aperiodic. The set $\{S^{-1}TS\,|\,S \in \textit{Aut}([0,1],\mathcal B,\mu)\}$ is dense in $\textit{Aut}([0,1],\mathcal B,\mu)$.
\end{lem}
\begin{proof}[Proof of \cref{5.lem:(2)Implies(3)}]
We divide the proof of \cref{5.lem:(2)Implies(3)} in the following steps: First, we will use the dense $G_\delta$ property of the sets $\mathcal G_{\vec \lambda}(\phi_1,...,\phi_\ell)$, $\vec \lambda\in [0,1]^\ell$, to find an increasing sequence $(m_k)_{k\in\N}$ in $\N$ and a  collection of weakly mixing transformations $S_{\vec \xi}$, $\vec \xi:=(\xi_1,...,\xi_\ell)\in [0,1]^\ell\cap \Q^\ell$, with the property that for any $A,B\in\mathcal B$ and every $j\in\{1,...,\ell\}$,
$$\lim_{k\rightarrow\infty}\mu(A\cap S_{\vec \xi}^{-\phi_j(m_k)}B)=(1-\xi_j)\mu(A\cap B)+\xi_j\mu(A)\mu(B).$$
Then, we will use the conjugacy lemma and the completeness of $\text{Aut}([0,1],\mathcal B,\mu)$ to approximate convenient weakly mixing transformations $T_{\vec \lambda}$, $\vec \lambda\in [0,1]^\ell$ which will turn out to satisfy \eqref{5.eq:LambdaMixingInLemma} for some subsequence $(n_k)_{k\in\N}$ of $(m_k)_{k\in\N}$.
\subsubsection{Finding the sequence $(m_k)_{k\in\N}$ and the transformations $S_{\vec \xi}$, $\vec \xi\in[0,1]^\ell\cap \Q^\ell$}
In order to find the sequence $(m_k)_{k\in\N}$ and the transformations $S_{\vec \xi}$, $\vec \xi\in[0,1]^\ell\cap \Q^\ell$ described above, we will utilize \cref{5.Lem:(1)implica(2)} applied to various subsequences of $(\phi_1(n),...,\phi_\ell(n))_{n\in\N}$. In turn, doing this  will require us to check that for any increasing sequence $(p_k)_{k\in\N}$ in $\N$, the sequences 
\begin{equation}\label{5.eq:ClaimOfAdequateSubsequence}
(\phi_j(p_k))_{k\in\N},\,j\in\{1,...,\ell\},
\end{equation}
are adequate. \\

\qedsymbol \textit{ For any increasing $(p_k)_{k\in\N}$ in $\N$, the sequences $\phi_1(p_k),...,\phi_\ell(p_k)$, $k\in\N$, are adequate.} 
Because \eqref{5.eq:AssumptionThatAllLimitsExsist} holds, in order to check that \eqref{5.eq:ClaimOfAdequateSubsequence} holds, it is enough to check that $\phi_1,...,\phi_\ell$ are adequate. To see that $\phi_1,...,\phi_\ell$ satisfy \eqref{0.eq:DivergenceConditionForAdequate}, pick $j\in\{1,...,\ell\}$ and let $\vec \lambda=(\lambda_1,...,\lambda_\ell)\in [0,1]^\ell$ be such that $\lambda_j=1$. Then, we can find a weakly mixing  $T\in\mathcal G_{\vec \lambda}(\phi_1,...,\phi_\ell)$ and an increasing sequence $(p_k)_{k\in\N}$ in $\N$ such that for any  measurable sets $A,B\in\mathcal B$, $$\lim_{k\rightarrow\infty}\mu(A\cap T^{-\phi_j(p_k)}B)=\mu(A)\mu(B).$$
Thus, we must have $\lim_{k\rightarrow\infty}|\phi_j(p_k)|=\infty$ which, by \eqref{5.eq:AssumptionThatAllLimitsExsist}, implies that \eqref{0.eq:DivergenceConditionForAdequate} holds. \\
To check that \eqref{0.eq:LinearCombinationConditionFroAdequate} holds, suppose for contradiction that there exists $a_1,...,a_\ell,b\in\Z$, $b\neq 0$, such that $\lim_{n\rightarrow\infty}\sum_{j=1}^\ell a_j\phi_j(n)=b$. Let $\vec \lambda=(0,...,0)$ and pick a weakly mixing $T\in \mathcal G_{\vec \lambda}(\phi_1,...,\phi_\ell)$. Since $T$ is aperiodic,  Rohlin's Lemma \cite[page
71]{halmosBooklectures} guarantees the existence of a set $A\in \mathcal B$  such that $T^{-b}A\cap A=\emptyset$ and $\mu(A)>0$.  Then, for some increasing sequence $(p_k)_{k\in\N}$ in $\N$, we have,
$$0=\mu(A\cap T^{-b}A)=\lim_{k\rightarrow\infty}\mu(A\cap T^{-\sum_{j=1}^\ell a_j\phi_j(p_k)}A)=\mu(A)>0,$$
a contradiction. Thus, \eqref{0.eq:LinearCombinationConditionFroAdequate} also holds and the sequences $\phi_1,...,\phi_\ell$ must be adequate.\\

\qedsymbol \textit{ Applying \cref{5.Lem:(1)implica(2)}.} Arguing as in the proof of (c)$\implies$(a) in \cref{5.Cor:AlgebraicChar}, we see that the fact that every set of the form $\mathcal G_{\vec \lambda}(\phi_1,...,\phi_\ell)$ is a dense $G_\delta$ set, implies that the sequences $\phi_1,...\phi_\ell$  satisfy \eqref{5.eq:QuasiIndependenceCondition}. Furthermore, since \eqref{5.eq:AssumptionThatAllLimitsExsist} holds, we have that for any increasing sequence $(p_k)_{k\in\N}$ in $\N$, the sequences $\psi_j(k)=\phi_j(p_k)$, $k\in\N$ and $j\in\{1,...,\ell\}$, satisfy $A(\phi_1,...,\phi_\ell)=A(\psi_1,...,\psi_\ell)$ and, so, they also satisfy \eqref{5.eq:QuasiIndependenceCondition}.\\
We now employ a diagonalization argument and repeated applications of \cref{5.Lem:(1)implica(2)} to  finer and finer subsequences of $(\phi_1(n),...,\phi_\ell(n))_{n\in\N}$ to obtain an increasing sequence $(m_k)_{k\in\N}$ in $\N$ with the property that for any $\vec \xi=(\xi_1,...,\xi_\ell)\in[0,1]^\ell\cap \Q^\ell$, there is a weakly mixing $S_{\vec \xi}\in \text{Aut}([0,1],\mathcal B,\mu)$ such that for any $A,B\in\mathcal B$ and any $j\in\{1,...\ell\}$,
\begin{equation}\label{5.eq:XiMixingS_xi}
\lim_{k\rightarrow\infty}\mu(A\cap S_{\vec \xi}^{-\phi_j(m_k)}B)=(1-\xi_j)\mu(A\cap B)+\xi_j\mu(A)\mu(B).
\end{equation}
\subsubsection{ Finding $(n_k)_{k\in\N}$ and defining the transformations $T_{\vec \lambda}$, $\vec \lambda\in[0,1]^\ell$}
For each $t\in\N$,  let $\mathcal Q_t=\{\frac{s}{2^t}\,|\,s\in\{0,...,2^t\}\}$. In order to prove that \eqref{5.eq:LambdaMixingInLemma} holds for each $\vec \lambda\in [0,1]^\ell$, we will pick an appropriate  conjugate for each term of the sequence  $S_{\vec \xi}$, $\vec \xi\in \bigcup_{t\in\N} \mathcal Q_t^\ell$, (each of which we will still denote by $S_{\vec \xi}$)  in such a way that for any given $\vec \lambda\in [0,1]^\ell$ there is a sequence $(\vec \xi_t)_{t\in\N}$ in $\bigcup_{t\in\N} \mathcal Q_t^\ell$
converging to $\vec \lambda$ with the property that the associated sequence $(S_{\vec \xi_t})_{t\in\N}$ is a Cauchy sequence with respect to $\partial$, and, so, by the completeness of $\text{Aut}([0,1],\mathcal B,\mu)$,  converges to some $T\in\text{Aut}([0,1],\mathcal B,\mu)$. As we will see, the convergence of $(S_{\vec \xi_t})_{t\in\N}$ to $T$ will be fast enough to ensure that $T$ can be taken to be $T_{\vec \lambda}$ in \eqref{5.eq:LambdaMixingInLemma} for some subsequence $(n_k)_{k\in\N}$ of $(m_k)_{k\in\N}$ (not depending on $\vec \lambda$) that we will define below.\\

\qedsymbol \textit{ Defining the sequence of conjugates of $S_{\vec \xi}$, $\vec \xi\in \bigcup_{t\in\N} \mathcal Q_t^\ell$.}
Starting with $t=1$ and then proceeding in an increasing order, we can replace each of the terms in  the finite sequences $S_{\vec \xi}$, $\vec \xi\in  \mathcal Q_t^\ell\setminus \mathcal Q_{t-1}^\ell$ (where $\mathcal Q_0=\emptyset$), by one of its conjugates to ensure that the resulting $\bigcup_{t\in\N} \mathcal Q_t^\ell$-indexed sequence of conjugates (which we will still denote by $S_{\vec \xi}$, $\vec \xi\in \bigcup_{t\in\N} \mathcal Q_t^\ell$)  satisfies the following conditions for each $t\in\N$:
\begin{enumerate}
    \item [(S.1)] Let $k_0=0$ and $\epsilon_0=1$. There exist an $\epsilon_t\in(0,\min\{\frac{1}{t},\epsilon_{t-1}\})$ and a  $k_t\in\N$ with $k_t>k_{t-1}$ such that for any $T\in \text{Aut}([0,1],\mathcal B,\mu)$ for which there is a $\vec \xi=(\xi_1,...,\xi_\ell)\in \mathcal Q_t^\ell$ with $\partial (T,S_{\vec \xi})<\epsilon_t$, one has that  for any $h,i\in\{1,...,t\}$ and any $j\in\{1,...,\ell\}$, 
    $$|\mu(A_h\cap T^{-\phi_j(m_{k_t})}A_i)-(1-\xi_j)\mu(A_h\cap A_i)-\xi_j\mu(A_h)\mu(A_i)|<\frac{1}{t}.$$
    Here, the sequence $(A_k)_{k\in\N}$ in $\mathcal B$ is the sequence used to define $\partial$ in \eqref{5.eq:DefnMetric}. So, in particular, $(A_k)_{k\in\N}$ is dense in $\mathcal B$.
    \item [(S.2)]  For each $\vec \xi\in \mathcal Q_{t+1}^\ell\setminus \mathcal Q_{t}^{\ell}$ and each $\vec \eta \in \mathcal Q_{t}^\ell$, if $\vec \xi \in (\vec \eta+[0,\frac{1}{2^{t}})^\ell)$, then 
    $$\partial (S_{\vec \xi},S_{\vec \eta})<\frac{\epsilon_{t}}{2^{t+1}}.$$ 
\end{enumerate}
We remark that (S.1) follows from \eqref{5.eq:XiMixingS_xi} and the fact that $\text{Aut}([0,1],\mathcal B,\mu)$ is a topological group. Statement  (S.2) is a consequence of \cref{5.lem:ConjugacyLem} applied to $S_{\vec \xi}$ (as opposed to $S_{\vec \eta}$). It is also worth noting that, roughly speaking, if \eqref{5.eq:LambdaMixingInLemma} holds for some $\vec \lambda\in [0,1]^\ell$, some  increasing sequence $(p_k)_{k\in\N}$ in $\N$, and some $T\in \text{Aut}([0,1],\mathcal B,\mu)$, \eqref{5.eq:LambdaMixingInLemma} will also hold for the same $\vec \lambda$, the same $(p_k)_{k\in\N}$, and any conjugate of $T$.\\

\qedsymbol \textit{ Defining the sequence $(n_k)_{k\in\N}$.}  Note that the sequence $(k_t)_{t\in\N}$ defined on (S.1) is a striclty increasing sequence and, so, $(m_{k_t})_{t\in\N}$ is a subsequence of $(m_k)_{k\in\N}$. We will let $(n_k)_{k\in\N}$ be the sequence defiend by $n_t=m_{k_t}$ for each $t\in\N$.\\

\qedsymbol\textit{ Formula \eqref{5.eq:LambdaMixingInLemma} holds for each $\vec \lambda \in \bigcup_{t\in\N}\mathcal Q_t^\ell$.} Let $t\in\N$ and let $\vec \lambda=(\lambda_1,...,\lambda_\ell)\in \mathcal Q_t^\ell$. Observe that in the inductive procedure used to pick an appropriate conjugate of the original  transformation $S_{\vec \lambda}$  satisfying conditions (S.1) and (S.2) described above, $S_{\vec \lambda}$ will not be modified after the $t$-th step. Thus, taking $T_{\vec \lambda}=S_{\vec \lambda}$, we see that $T_{\vec \lambda}$ is weakly mixing and that \eqref{5.eq:LambdaMixingInLemma} holds.\\

\qedsymbol\textit{ Formula \eqref{5.eq:LambdaMixingInLemma} holds for each $\vec \lambda \in [0,1]^\ell\setminus\bigcup_{t\in\N}\mathcal Q_t^\ell$.} Let  $\vec \lambda=(\lambda_1,...,\lambda_\ell) \in [0,1]^\ell\setminus\bigcup_{t\in\N}\mathcal Q_t^\ell$. For each $k\in\N$, let $\vec \xi_k$ be the unique element of  $\mathcal Q_k^\ell$ with the property that $\vec \lambda \in (\vec \xi_k+[0,\frac{1}{2^k})^\ell)$. Note that $\lim_{k\rightarrow\infty}\vec \xi_k=\vec \lambda$ and that for any $k<k'$, $\vec \xi_{k'}\in(\vec \xi_k+[0,\frac{1}{2^k})^\ell)$. Thus, by (S.2) and since $(\epsilon_t)_{t\in\N}$ decreases to zero, for any $t\in\N$ and any  $k>t$, 
\begin{equation}\label{5.eq:Cauchyness}
\partial(S_{\vec \xi_k},S_{\vec \xi_{t}})=\sum_{r=t}^{k-1}\partial(S_{\vec \xi_{r+1}},S_{\vec \xi_r})
< \sum_{r=1}^\infty \frac{\epsilon_t}{2^{r+1}}=\frac{\epsilon_{t}}{2}<\frac{1}{t}.
\end{equation} 
It follows that the sequence $(S_{\vec \xi_k})_{k\in\N}$ is a Cauchy sequence in $\text{Aut}([0,1],\mathcal B,\mu)$. We set $$T:=\lim_{k\rightarrow\infty} S_{\vec \xi_k}$$
in the weak topology of $\text{Aut}([0,1],\mathcal B,\mu)$. By \eqref{5.eq:Cauchyness}, for each $t\in\N$, 
$$\partial(S_{\vec \xi_t},T)\leq \limsup_{k\rightarrow\infty}\Big(\partial(S_{\vec \xi_t},S_{\vec \xi_k})+\partial(S_{\vec \xi_k},T)\Big)\leq \frac{\epsilon_t}{2}.
$$
To prove that we can take $T_{\vec \lambda}=T$ in \eqref{5.eq:LambdaMixingInLemma}, let $A,B\in\mathcal B$, let $\epsilon>0$, and let $h,i\in\N$ be such that $\mu(A\triangle A_h)<\frac{\epsilon}{4}$ and $\mu(B\triangle A_i)<\frac{\epsilon}{4}$. Letting $\vec \xi_k=(\xi_{k,1},...,\xi_{k,\ell})$ for each $k\in\N$, we see that for every $j\in\{1,...,\ell\}$, (S.1) implies that,
\begin{multline*}
    \limsup_{k\rightarrow\infty}|\mu(A\cap T^{-\phi_j(n_k)}B)-(1-\lambda_j)\mu(A\cap B)-\lambda_j\mu(A)\mu(B)|\\
    \leq \limsup_{k\rightarrow\infty}|\mu(A_h\cap T^{-\phi_j(n_k)}A_i)-(1-\lambda_j)\mu(A_h\cap A_i)-\lambda_j\mu(A_h)\mu(A_i)|+\epsilon\\
    =\limsup_{k\rightarrow\infty} |\mu(A_h\cap T^{-\phi_j(n_k)}A_i)-(1-\xi_{k,j})\mu(A_h\cap A_i)-\xi_{k,j}\mu(A_h)\mu(A_i)|+\epsilon\\
    \leq\limsup_{k\rightarrow\infty}(\frac{1}{k}+\epsilon)=\epsilon.
\end{multline*}
Finally, since $\vec 0 \in\bigcup_{t\in\N}\mathcal Q_t^\ell$ and, by assumption, $\vec \lambda=(\lambda_1,...,\lambda_\ell) \in [0,1]^\ell\setminus\bigcup_{t\in\N}\mathcal Q_t^\ell$, $\lambda_j\neq 0$ for at least one $j\in\{1,...,\ell\}$. By \eqref{5.eq:LambdaMixingInLemma} it follows that  for every $A,B\in \mathcal B$, 
$$\lim_{k\rightarrow\infty}\mu(A\cap T_{\vec \lambda}^{-\phi_j(n_k)}B)=(1-\lambda_j)\mu(A\cap B)+\lambda_j\mu(A)\mu(B),$$
which yields that   $T_{\vec \lambda}$ has no non-trivial eigenfunctions and, so, is  weak mixing. 
We are done. 
\end{proof}
\subsection{The proof of \cref{0.thm:GenericInterpolationResult}}
\begin{proof}[Proof of \cref{0.thm:GenericInterpolationResult}]
That (1)$\implies$(2) follows from \cref{5.Lem:(1)implica(2)}. That (2)$\implies$(3) follows from \cref{5.lem:(2)Implies(3)}. That (3)$\implies$(1) follows from the proof of \eqref{5.eq:ClaimOfAdequateSubsequence} (see Subsubsection 5.3.1) and  (c)$\implies$(a) in \cref{5.Cor:AlgebraicChar}.
\end{proof}
\section{Applications to the theory of \rm{IP$^*$}-recurrence}
As we mentioned in the Introduction, Subsubsections 1.3.2 and  1.3.3, \cref{0.thm:MainResult} item (iii) can be employed to obtain interesting examples in ergodic theory. In this section, we provide three of these examples (Corollaries  \ref{0.cor:IndependentPolyFailure}, \ref{0.cor:Example2}, and \ref{0.cor:Example3} in the Introduction). The proofs of  these results will require the use of a variant of (ii)$\implies$(iii) in \cref{0.thm:MainResult} which we will state and prove in the next subsection and which deals with the concept of IP-convergence.
\subsection{IP-convergence for unitary operators and a variant of (ii)$\implies$(iii) in \cref{0.thm:MainResult}}
Denote the set of all finite, non-empty subsets of $\N$ by $\mathcal F$ and let $(X,d)$ be a compact metric space. For any $\mathcal F$-indexed sequence $(x_\alpha)_{\alpha\in\mathcal F}$, we write 
$$\mathop{\text{IP-lim}}_{\alpha\in\mathcal F}x_\alpha=x$$
if there exists an $x\in X$ with the property that  for every $\epsilon>0$ there exists $k_\epsilon\in\N$ such that if $\min \alpha>k_\epsilon$, then $d(x_\alpha,x)<\epsilon$. We remark in passing that, by Hindman's theorem \cite{HIPPartitionRegular}, for any $\mathcal F$-indexed sequence $(x_\alpha)_{\alpha\in\mathcal F}$ in $X$, there is a sequence of disjoint, non-empty, and finite subsets of $\N$, $(\beta_k)_{k\in\N}$, for which the $\mathcal F$-indexed sequence $y_\alpha:=x_{\bigcup_{k\in\alpha}\beta_k}$, $\alpha\in\mathcal F$, satisfies 
$$\mathop{\text{IP-lim}}_{\alpha\in\mathcal F}y_\alpha=x$$
for some $x\in X$. \\
The following is a consequence of (the proof of)
\cite[Proposition 1.7]{BKLUltrafilterPoly}, \cite[Lemma 2.5]{BKLUltrafilterPoly},  and the equivalence between IP-limits and convergence along idempotent ultrafilters (see \cite[p.36]{ERTaU} or \cite[Lemma 1.1]{BKLUltrafilterPoly}, for example). We omit the proof. For any sequence $(n_k)_{k\in\Z}$ in $\Z$ and any $\alpha\in\mathcal F$, we will let $n_\alpha=\sum_{k\in\alpha}n_k$.
\begin{lem}\label{6.lem:LiftingIPConvergence}
Let $D,N\in\N$ be such that $D\leq N$ and let $U:\mathcal H\rightarrow\mathcal H$ be a unitary operator. Let $(n_k)_{k\in\N}$ be an increasing sequence in $\N$ with the properties that for each $k\in\N$, $k!|n_k$, for every $a_1,...,a_N\in\Z$,
$$\lim_{k\rightarrow\infty}U^{\sum_{j=1}^Na_jn_k^j}\in\{\text{Id},0\}$$
 in the weak operator topology, and  $\lim_{k\rightarrow\infty}U^{\sum_{j=1}^Na_jn_k^j}=0$ whenever $\deg(\sum_{j=1}^Na_jx^j)\geq D$. Furthermore, assume that for each $j\in\{1,...,D\}\setminus \{D\}$, there is an $r_j\in\N$ for which 
$$\lim_{k\rightarrow\infty}U^{r_jn_k^j}=\text{Id}.$$
Then, there exists an increasing sequence $(m_k)_{k\in\N}$ in $\N$ with the property that for any $a_1,...,a_N\in\Z$,
$$\mathop{\text{IP-lim}}_{\alpha\in\mathcal F}U^{\sum_{j=1}^Na_jm_\alpha^j}=\lim_{k\rightarrow\infty}U^{\sum_{j=1}^Na_jn_k^j}$$
in the weak operator topology.
\end{lem}
We are now in position to state the variant of (ii)$\implies$(iii) in \cref{0.thm:MainResult} dealing with IP-convergence.
\begin{thm}\label{6.thm:IPRefinementOf(iii)}
    Let $d,N\in\N$ be such that $d <N$ and let $G$ be a finite index subgroup of $\Z^d$. 
    Then, there exist a Borel probability measure $\sigma$ on $\mathbb T$ and an increasing sequence $(n_k)_{k\in\N}$ in $\N$ with the property that for any continuous $f:\mathbb T^N\rightarrow \mathbb C$ and any measurable $E\subseteq \mathbb T$, 
    \begin{equation}\label{6.eq:IPWeakStarConvergence}
    \mathop{\text{IP-lim}}_{\alpha\in\mathcal F}\int_\mathbb T \mathbbm 1_E(x)f(n_\alpha x,n_\alpha^2x,...,n_\alpha^Nx)\text{d}\sigma(x)=\sigma(E)\int_{\mathbb T^N}f(y_1,...,y_N)\text{d}\lambda_{G\times \{\vec 0\}}(y_1,...,y_N),
    \end{equation}
    where $\vec 0\in \Z^{N-d}$.
\end{thm}
\begin{proof}
    Consider the adequate sequences $\phi_1,...,\phi_N:\N\rightarrow \Z$ definde by $\phi_j(n)=(n!)^j$, $j\in\{1,...,N\}$, and note that $A(\phi_1,...,\phi_N)=\{\vec 0\}$. By item (iii) in \cref{0.thm:MainResult} (or, rather, by \eqref{0.eq:Mixing-RigidCharacterization}), we can find a subsequence $(p_k)_{k\in\N}$ of $(k!)_{k\in\N}$ and  a unitary operator $U:\mathcal H\rightarrow\mathcal H$ with the property that for any $a_1,...,a_N\in\Z$,
    $$\lim_{k\rightarrow\infty}U^{\sum_{j=1}^Na_jp_k^j}=\begin{cases}
        \text{Id},\,\text{ if }(a_1,...,a_N)\in G\times\{\vec 0\},\\
        0,\text{ if }(a_1,...,a_N)\not\in G\times\{\vec 0\}.
    \end{cases}
    $$
    By \cref{6.lem:LiftingIPConvergence} applied with $D=d+1$, it follows that there exists an increasing sequence  $(n_k)_{k\in\N}$ in $\N$ with the property that 
    \begin{equation}\label{6.eq:IPConvForUnitaryInproof}
    \mathop{\text{IP-lim}}_{\alpha\in\mathcal 
 F}U^{\sum_{j=1}^Na_jn_\alpha^j}=\begin{cases}
        \text{Id},\,\text{ if }(a_1,...,a_N)\in G\times\{\vec 0\},\\
        0,\text{ if }(a_1,...,a_N)\not\in G\times\{\vec 0\},
    \end{cases}
    \end{equation}
for every $a_1,...,a_N\in\Z$.\\
Let $f\in\mathcal H$ be such that $\|f\|_{\mathcal H}=1$. By Bochner's theorem there is a Borel probability measure $\sigma$ on $\mathbb T$ with the property that for each $n\in\Z$, 
$$\langle U^nf,f \rangle=\int_\mathbb T e^{2\pi i nx}\text{d}\sigma(x).$$  
So, in particular, \eqref{6.eq:IPConvForUnitaryInproof} implies that for any measurable $E\subseteq \mathbb T$ and any $a_1,...,a_N\in\Z$,
$$
 \mathop{\text{IP-lim}}_{\alpha\in\mathcal 
 F}\int_\mathbb T \mathbbm 1_E(x)e^{2\pi i(\sum_{j=1}^Na_jn_\alpha^j)x}\text{d}\sigma(x)=\begin{cases}
        \sigma(E),\,\text{ if }(a_1,...,a_N)\in G\times\{\vec 0\},\\
        0,\text{ if }(a_1,...,a_N)\not\in G\times\{\vec 0\},
        \end{cases}
$$
which is equivalent to \eqref{6.eq:IPWeakStarConvergence}. We are done. 
\end{proof}
\subsection{The proof of \cref{0.cor:IndependentPolyFailure}}
\begin{proof}[Proof of \cref{0.cor:IndependentPolyFailure}] The proof of \cref{0.cor:IndependentPolyFailure} consists of two major steps. First, we will use \cref{6.thm:IPRefinementOf(iii)} to find an increasing sequence $(n_k)_{k\in\N}$ in $\N$ and a Borel probability measure $\sigma$ on $\mathbb T$ with \textit{ad-hoc} properties. Then, we will use the probability measure $\sigma$ to define the invertible probability preserving  system $(X,\mathcal A,\nu,T)$ for which \cref{0.cor:IndependentPolyFailure} holds. \\

\qedsymbol \textit{ Finding $\sigma$ and $(n_k)_{k\in\N}$.}
Let $p_1,...,p_\ell\in\Z[x]$ be linearly independent polynomials with zero constant term and let $D=\max_{1\leq j\leq \ell}\deg(p_j)$. For each $j\in\{1,...,\ell\}$, let $\vec c_j=(c_{j,1},...,c_{j,D})\in\Z^D\setminus\{\vec 0\}$ be such that 
$$p_j(x)=\sum_{s=1}^D c_{j,s}x^s.$$
Let 
$$G':=\{a_0(-2\vec c_1+\vec c_2)+\sum_{j=1}^\ell a_j3\vec c_j\,|\,a_0,...,a_\ell\in\Z\}.$$
If $G'$ has finite index in $\Z^D$, we set $G=G'$. If not, let $\vec d_1,...,\vec d_{D-\ell}$ be linearly independent vectors in $\Z^D$ with the additional property that for any $j\in\{1,...,\ell\}$ and any $s\in\{1,...,D-\ell\}$, $\langle\vec c_j,\vec d_s \rangle=0$ (such vectors can be obtained by taking a rational basis of the orthogonal complement of $\text{span}_\Q\{\vec c_1,...,\vec c_\ell\}$ and then multiplying each member of this basis by an appropriate integer). We set  
$$G=\{\vec c+\sum_{s=1}^{D-\ell} a_s3 \vec d_s\,|\,\vec c\in G'\text{ and }a_1,...,a_{D-\ell}\in\Z\}$$
Notice that, in either case, we obtain that (G.1) $M\Z^D\subseteq G$ for some $M\in\N$ and, so, that $G$ has finite index in $\Z^D$ and (G.2) $\sum_{j=1}^\ell a_j\vec c_j\in G'$ if and only if $\sum_{j=1}^\ell a_j\vec c_j\in G$.\\
Applying \cref{6.thm:IPRefinementOf(iii)} in the case that $d=D$ and $N>D$ is arbitrary, we see that there  exist an increasing sequence $(n_k)_{k\in\N}$ in $\N$ and a Borel probability measure $\sigma$ on $\mathbb T$ such that 
 for every $a_1,...,a_\ell \in\Z$,
$$\mathop{\text{IP-lim}}_{\alpha\in\mathcal F}\int_\mathbb T e^{2\pi i\sum_{j=1}^\ell a_jp_j(n_\alpha)x}\text{d}\sigma(x)=\begin{cases}
    1,\text{ if }\sum_{j=1}^\ell a_j\vec c_j\in G,\\
    0, \text{ if }\sum_{j=1}^\ell a_j\vec c_j\not\in G.
\end{cases}$$
Invoking (G.2) above and setting 
$$H=\{(a_1,...,a_\ell)\in \Z^\ell\,|\,\sum_{j=1}^\ell a_j\vec c_j\in G'\},$$
it follows that 
\begin{equation}\label{6.eq:KewWeakLimitForLIpolys}
\mathop{\text{IP-lim}}_{\alpha\in\mathcal F}\int_\mathbb T f(p_1(n_\alpha)x,...,p_\ell(n_\alpha)x)\text{d}\sigma(x)=\int_{\mathbb T^\ell}f(y_1,...,y_\ell)\text{d}\lambda_{H}(y_1,...,y_\ell)
\end{equation}
for any continuous function $f:\mathbb T^\ell\rightarrow \mathbb C$. Observe that, when we identify $\mathbb T^\ell$ with $[0,1)^\ell$,  $\lambda_H$ corresponds to the normalized Haar measure on 
\begin{equation}\label{6.eq:DefnH}
\{\left(a_0(\frac{1}{3}\vec e_1+\frac{2}{3}\vec e_2)+\sum_{s=3}^\ell a_s\frac{1}{3}\vec e_s\right)
\mod 1\,|\,a_0,a_3,...,a_\ell\in \{0,1,2\}\}.
\end{equation}
\qedsymbol \textit{ Defining $(X,\mathcal A,\nu,T)$.} Set $X=[0,1)^2$, $\mathcal A=\text{Borel}([0,1)^2)$, and let $\nu:=\sigma\times \mu$. Define $T:X\rightarrow X$ by $T(x,y)=(x,y+x)\mod 1$. Clearly, $T$ is an invertible, $\nu$ preserving transformation. Consider the set $A=\mathbb T\times [0,2/3)$ and let $\epsilon=\frac{1}{3^{\ell+1}}$.\\
\qedsymbol \textit{ Formula \eqref{0.eq:LargeReturns} holds with $A=\mathbb T\times [0,2/3)$ and  $\epsilon=\frac{1}{ 3^{\ell+1}}$.}  Observe that,  by \eqref{6.eq:KewWeakLimitForLIpolys} and the (proof of the) continuous mapping theorem (see \cite[Remarks in page 149]{waltersIntroduction}, for example),  for $\mu$-a.e. $y\in [0,1)$,
$$
\mathop{\text{IP-lim}}_{\alpha\in\mathcal F}\int_\mathbb T \prod_{j=1}^\ell \mathbbm 1_{[0,2/3)}(y+p_j(n_\alpha)x)\text{d}\sigma(x)=\int_{\mathbb T^\ell}\prod_{j=1}^\ell\mathbbm 1_{[0,2/3)}(y+y_j)\text{d}\lambda_{H}(y_1,...,y_\ell).
$$
It then follows from the definition of $\nu$, $T$, and $A$ that 
\begin{multline}
    \mathop{\text{IP-lim}}_{\alpha\in\mathcal F} \nu(A\cap T^{-p_1(n_\alpha)}A\cap \cdots\cap T^{-p_\ell(n_\alpha)}A)\\
    =\mathop{\text{IP-lim}}_{\alpha\in\mathcal F}\int_\mathbb T\int_\mathbb T \mathbbm 1_{A}(x,y)\prod_{j=1}^\ell \mathbbm 1_{A}(x,y+p_j(n_\alpha)x)\text{d}\sigma(x)\text{d}\mu(y)\\
    =\mathop{\text{IP-lim}}_{\alpha\in\mathcal F}\int_\mathbb T\int_\mathbb T \mathbbm 1_{[0,2/3)}(y)\prod_{j=1}^\ell \mathbbm 1_{[0,2/3)}(y+p_j(n_\alpha)x)\text{d}\sigma(x)\text{d}\mu(y)\\
    =\int_\mathbb T\int_{\mathbb T^\ell}\mathbbm 1_{[0,2/3)}(y)\prod_{j=1}^\ell\mathbbm 1_{[0,2/3)}(y+y_j)\text{d}\lambda_{H}(y_1,...,y_\ell)\text{d}\mu(y)
\end{multline}
Thus, by \eqref{6.eq:DefnH},
\begin{multline}
 \mathop{\text{IP-lim}}_{\alpha\in\mathcal F} \nu(A\cap T^{-p_1(n_\alpha)}A\cap \cdots\cap T^{-p_\ell(n_\alpha)}A)\\
    =\frac{1}{3^{\ell-1}}\sum_{n=0}^2\sum_{s_3,...,s_\ell=0}^2
    \int_\mathbb T \mathbbm 1_{[0,2/3)}(y)\mathbbm 1_{[0,2/3)}(\frac{n}{3}+y)\mathbbm 1_{[0,2/3)}(\frac{2n}{3}+y)\prod_{j=3}^\ell\mathbbm 1_{[0,2/3)}(\frac{s_j}{3}+y)\text{d}\mu(y)\\
=\frac{2^{\ell-2}
}{3^{\ell-1}}\sum_{n=0}^2
    \int_\mathbb T \mathbbm 1_{[0,2/3)}(y)\mathbbm 1_{[0,2/3)}(\frac{n}{3}+y)\mathbbm 1_{[0,2/3)}(\frac{2n}{3}+y)\text{d}\mu(y)\\
    =\frac{2^{\ell-2}
}{3^{\ell-1}}\left( \frac{2}{3}+0+0\right)=\frac{2^{\ell-1}}{3^\ell}.
\end{multline}
Since
$$\frac{2^{\ell+1}}{3^{\ell+1}}-\frac{2^{\ell-1}}{3^\ell}=\frac{2^{\ell-1}}{3^{\ell+1}}\geq 2\epsilon,$$
the definition of IP-limit allows us to find a $k_0\in\N$ such that 
$$
\{\sum_{s=1}^tn_{k_s}\,|\,k_0<k_1<\cdots<k_t,\,t\in\N\}\cap \{n\in\Z\,|\,\nu(A\cap T^{-p_1(n)}A\cap \cdots\cap T^{-p_\ell(n)}A)>\nu^{\ell+1}(A)-\epsilon\}=\emptyset,
$$
which proves that \eqref{0.eq:LargeReturns} holds. We are done. 
\end{proof}
\subsection{The proof of \cref{0.cor:Example2}}
In order to prove \cref{0.cor:Example2} we will make use of the following version of  \cite[Theorem 2.1]{BHKNilSystems2005}, which in turn is proved with help of the results in  \cite{BehrendSetsWithNo3-APs}.
\begin{thm}\label{6.thm:Behrend}
     For each $\ell\in \N$ there exists a measurable subset $B$ of $[0,1)$ which is the union of disjoint subintervals of $[0,1)$ and satisfies 
    $$\int_\mathbb T\int_\mathbb T \mathbbm 1_B(y)\mathbbm 1_B(y+z)\mathbbm 1_B(y+2z)\text{d}\mu(y)\text{d}\mu(z)\leq \frac{\mu^\ell(B)}{2}.$$
    (Observe that $\mu(B)>0$).
\end{thm}
\begin{proof}[Proof of \cref{0.cor:Example2}]
Let $N=\deg(p)=\deg(q)$ and let $d=\deg(p)-1$. By \cref{6.thm:IPRefinementOf(iii)}, we can find an increasing sequence $(n_k)_{k\in\N}$ in $\N$ and a Borel probability measure $\sigma$ on $\mathbb T$ with the property that for any continuous function $f:\mathbb T^N\rightarrow \mathbb C$, 
\begin{equation}\label{6.eq:ProofOfExample2}
\mathop{\text{IP-lim}}_{\alpha\in\mathcal F}\int_\mathbb T f(n_\alpha x,n_\alpha^2x,...,n_\alpha^Nx)\text{d}\sigma(x)=\int_{\mathbb T^N}f(0,...,0,y_N)\text{d}\mu(y_N).
\end{equation}
Since $\deg(2p-q)<N$, \eqref{6.eq:ProofOfExample2} implies that for any continuous $f:\mathbb T^2\rightarrow \mathbb C$,
\begin{equation*}
    \mathop{\text{IP-lim}}_{\alpha\in\mathcal F}\int_\mathbb T f(p(n_\alpha)y,q(n_\alpha)y)\text{d}\sigma(y)
    =\mathop{\text{IP-lim}}_{\alpha\in\mathcal F}\int_\mathbb T f(p(n_\alpha)y,2p(n_\alpha)y)\text{d}\sigma(y)
    =
    \int_{\mathbb T}f(t,2t)\text{d}\mu(t).
\end{equation*}
Let $X=[0,1)^2$, let $\mathcal A=\text{Borel}([0,1)^2)$, let $\nu=\sigma\times \mu$ and let $T:[0,1)^2\rightarrow [0,1)^2$ be defined by $T(x,y)=(x,y+x)$. Fix $\ell\in\N$ and let $B\subseteq[0,1)$ be the set guaranteed to exist by \cref{6.thm:Behrend}. Setting $A=\mathbb T\times B$, we obtain by the continuous mapping theorem that,
\begin{multline*}
    \mathop{\text{IP-lim}}_{\alpha\in\mathcal F} \nu(A\cap T^{-p(n_\alpha)}A\cap T^{-q(n_\alpha)}A)\\
    = \mathop{\text{IP-lim}}_{\alpha\in\mathcal F}\int_\mathbb T\int_\mathbb T \mathbbm 1_B(y)\mathbbm 1_B(y+p(n_\alpha)x)\mathbbm 1_B(y+q(n_\alpha)x)\text{d}\sigma(x)\text{d}\mu(y)\\
    =\int_\mathbb T\int_\mathbb T \mathbbm 1_B(y)\mathbbm 1_B(y+z)\mathbbm 1_B(y+2z)\text{d}\mu(z)\text{d}\mu(y)\leq \frac{\mu^\ell(B)}{2}.
\end{multline*}
Since $\mu(B)=\nu(A)$, the result follows.
\end{proof}
\subsection{The proof of \cref{0.cor:Example3}}
\begin{proof}[Proof of \cref{0.cor:Example3}]
Let $(p_t)_{t\in\N}$ be an increasing sequence of prime numbers. By \cref{6.thm:IPRefinementOf(iii)}, we have that for each $t\in\N$ there is a Borel probability measure  $\sigma_t$ on $\mathbb T$ and an increasing sequence $(n_k^{(t)})_{k\in\N}$ in $\N$ such that for  any continuous function $f:\mathbb T^3\rightarrow \mathbb C$, 
 $$
 \mathop{\text{IP-lim}}_{\alpha\in\mathcal F}\int_\mathbb T f(n_\alpha^{(t)}y,2n_\alpha^{(t)}y,(n_\alpha^{(t)})^2y)\text{d}\sigma_t(y)\\
    =\int_{\mathbb T^3}f(y_1,2y_1,y_2)\text{d}\lambda_{p_t\Z\times\{0\}}(y_1,y_2).$$
Fix $\ell\in\N$. Let $S:[0,1)^2\rightarrow [0,1)^2$ be given by $S(x,y)=(x,y+x)\mod 1$ and for each $t\in\N$, let $\nu_t=\sigma_t\times \mu$. Arguing as in the proof of \cref{0.cor:Example2}, we obtain that for some Borel measurable $B\subseteq [0,1)$, the set $A=\mathbb T\times B$ satisfies $\nu_t(A)=\mu(B)>0$ and 
\begin{multline}\label{6.eq:ProofExample3}
\limsup_{t\rightarrow\infty}\mathop{\text{IP-lim}}_{\alpha\in\mathcal F} \nu_t(A\cap S^{-n_\alpha^{(t)}}A\cap S^{-2n_\alpha^{(t)}}A\cap S^{-(n_\alpha^{(t)})^2}A)\\
    = \limsup_{t\rightarrow\infty}\mathop{\text{IP-lim}}_{\alpha\in\mathcal F}\int_\mathbb T\int_\mathbb T \mathbbm 1_B(y)\mathbbm 1_B(y+n_\alpha^{(t)}x)\mathbbm 1_B(y+2n_\alpha^{(t)}x)\mathbbm 1_B(y+(n_\alpha^{(t)})^2x)\text{d}\sigma_t(x)\text{d}\mu(y)\\
    =\limsup_{t\rightarrow\infty}\int_\mathbb T\int_{\mathbb T^2} \mathbbm 1_B(y)\mathbbm 1_B(y+y_1)\mathbbm 1_B(y+2y_1)\mathbbm 1_B(y+y_2)\text{d}\lambda_{p_t\Z}(y_1)\text{d}\mu(y_2)\text{d}\mu(y)\\
    =\int_\mathbb T\int_{\mathbb T^2} \mathbbm 1_B(y)\mathbbm 1_B(y+y_1)\mathbbm 1_B(y+2y_1)\mathbbm 1_B(y+y_2)\text{d}\mu(y_1)\text{d}\mu(y_2)\text{d}\mu(y)\leq \frac{\mu^\ell(B)}{2}\mu(B).
\end{multline}
Let $X=\prod_{t\in\N} [0,1)^2$, let $\mathcal A=\text{Borel}(X)$, let $\nu=\prod_{t\in\N}\nu_t$, and let $T:X\rightarrow X$ be given by $[T(\omega)](t)=S(\omega(t))$. For each $t\in\N$, let $\pi_t:X\rightarrow [0,1)^2$ be given by $\pi_t(\omega)=\omega(t)$. By \eqref{6.eq:ProofExample3}, we have that for every $\ell\in\N$ there is a $t\in\N$ and a Borel measurable $A\subseteq [0,1)^2$ with $\nu(\pi_t^{-1}A)>0$ for which the set 
\begin{multline}\label{6.eq:ProofOfEx3Sets}
\{n\in\Z\,|\,\nu(\pi^{-1}_tA\cap T^{-n}(\pi^{-1}_tA)\cap T^{-2n}(\pi^{-1}_tA)\cap T^{-n^2}(\pi^{-1}_tA))>\nu^\ell(\pi^{-1}_tA)\}\\
=\{n\in\Z\,|\,\nu_t(A\cap S^{-n}A\cap S^{-2n}A\cap S^{-n^2}A)>\nu_t^\ell(A)\}
\end{multline}
is not an \rm{IP$^*$} set. We remark that the equality in \eqref{6.eq:ProofOfEx3Sets} follows from the fact that $\pi_t\circ T=S\circ \pi_t$ for each $t\in\N$. We are done. 
\end{proof}

\bibliographystyle{abbrvnat}
\bibliography{Bib.bib}

\end{document}